\documentclass[11pt,reqno,a4paper]{amsart}
\pdfoutput=1

\usepackage[margin=1.05in,footskip=0.15in,tmargin=1in,bmargin=1in]{geometry}
\usepackage[indent,skip=.2\baselineskip]{parskip}

\usepackage[utf8]{inputenc}
\usepackage[english]{babel}
\usepackage{amsthm}
\usepackage{amsfonts}
\usepackage{graphicx}
\usepackage[font=footnotesize]{caption}
\usepackage{floatrow}
\usepackage{accents}
\usepackage{thmtools} 
\usepackage{amsmath}
\usepackage{array}
\usepackage{amssymb}
\usepackage{bm}

\DeclareTextFontCommand{\bfemph}{\bfseries\em}
\newcommand{\term}{\bfemph}

\usepackage{tikz}
\usepackage{fancyhdr}
\usepackage{comment}
\usepackage{enumitem}
\usepackage{mathtools}
\usepackage[sort&compress,comma,square,numbers]{natbib}

\usepackage[version=4]{mhchem}

\usepackage[ruled,vlined]{algorithm2e}

\numberwithin{equation}{section}

\usepackage[unicode,psdextra,pdfusetitle,hypertexnames=false,
bookmarks, bookmarksdepth=2, colorlinks=true, linkcolor=blue!50!black, citecolor=black!70, urlcolor=blue!50!black]{hyperref}
\usepackage[capitalise, noabbrev,nameinlink]{cleveref} 

\newtheorem{theorem}{Theorem}[section]
\newtheorem{proposition}[theorem]{Proposition}

\newtheorem{corollary}[theorem]{Corollary}
\newtheorem{lemma}[theorem]{Lemma}

\newtheorem{theoremalphabetic}{Theorem}

\theoremstyle{definition}
\newtheorem{definition}[theorem]{Definition}
\newtheorem{example}[theorem]{Example}
\newtheorem{remark}[theorem]{Remark}

\crefname{theorem}{Theorem}{Theorems}
\Crefname{theorem}{Theorem}{Theorems}
\crefname{proposition}{Proposition}{Propositions}
\Crefname{proposition}{Proposition}{Propositions}
\crefname{notation}{Notation}{Notations}
\Crefname{notation}{Notation}{Notations}
\crefname{corollary}{Corollary}{Corollaries}
\Crefname{corollary}{Corollary}{Corollaries}
\crefname{lemma}{Lemma}{Lemmas}
\Crefname{lemma}{Lemma}{Lemmas}
\crefname{definition}{Definition}{Definitions}
\Crefname{definition}{Definition}{Definitions}
\crefname{example}{Example}{Examples}
\Crefname{example}{Example}{Examples}
\crefname{remark}{Remark}{Remarks}
\Crefname{remark}{Remark}{Remarks}
\crefname{theoremalphabetic}{Theorem}{Theorems}
\Crefname{theoremalphabetic}{Theorem}{Theorems}

\newcommand{\mathdash}{\relbar\mkern-9mu\relbar}

\newcommand{\R}{\mathbb{R}}
\newcommand{\Q}{\mathbb{Q}}

\newcommand{\Z}{\mathbb{Z}}
\newcommand{\mO}{\mathcal{O}}
\newcommand{\mS}{\mathcal{S}}
\newcommand{\mD}{\mathcal{D}}
\newcommand{\mM}{\mathcal{M}}
\newcommand{\mC}{\mathcal{C}}
\newcommand{\mE}{\mathcal{E}}

\newcommand{\mG}{\mathcal{G}}
\newcommand{\wQ}{\widetilde{Q}}
\newcommand{\oP}{\bar{P}}

\newcommand{\bs}{\bar{s}}
\newcommand{\bl}{\bar{\ell}}

\DeclareMathOperator{\im}{Im}

\DeclareMathOperator{\sign}{sign}

\DeclareMathOperator{\diag}{diag}

\DeclareMathOperator{\id}{id}

\newcommand{\Addresses}{{
  \bigskip
  \footnotesize
  \textsc{Carles Checa, University of Copenhagen}\par\nopagebreak
  \textit{E-mail address}, \texttt{ccn@math.ku.dk}

  \medskip

 \textsc{Elisenda Feliu, University of Copenhagen}\par\nopagebreak
  \textit{E-mail address},  \texttt{efeliu@math.ku.dk}
}}

\begin{document}

 \fancyhead[CO]{\scriptsize C. Checa and  E. Feliu}
 \fancyhead[CE]{\scriptsize C. Checa and E. Feliu}
\fancyhead[LE]{\scriptsize\thepage}
\fancyhead[RO]{\scriptsize\thepage}
\fancyhead[LO]{}
\fancyhead[RE]{}
\fancyfoot[R]{}
\fancyfoot[C]{}
\fancyfoot[L]{}
\thispagestyle{empty}

\title[Multiple positive zeros]{An effective criterion for multiple positive zeros of vertically parametrized polynomial systems}

\author{Carles Checa and Elisenda Feliu}

\date{\today}


\begin{abstract}
We present an effective criterion for determining whether an (augmented) vertically parametrized polynomial system admits multiple positive zeros for some choice of parameter values. Our method builds on previous algorithms from chemical reaction network theory and reduces the problem to checking the feasibility of  linear systems of equalities and inequalities. 
Our criterion provides a necessary condition for the existence of multiple positive zeros that applies to any augmented vertically parametrized polynomial system, and we show that when the kernel of the coefficient matrix of the system displays a particular sparsity structure, this condition also becomes sufficient. This provides  a full characterization of the existence of multiple zeros for this type of system.
\end{abstract}


\maketitle

\section*{Introduction}

Polynomial systems of equations that arise in applications often have  fixed support and their coefficients are functions of some parameters. \term{Vertically parametrized}  systems (vertical systems for short) \cite{genericrootcounts,FELIU2025630} constitute an example of this and arise naturally when describing the steady states of chemical reaction networks \cite{dickenstein2016biochemical, feinberg-book,feliu2025genericgeometrysteadystate}, finding the critical points of hypersurfaces  \cite{gkz1994}, and in geometric modelling \cite{craciunsottile}. These systems are defined by   Laurent polynomials of the  form
\begin{equation*}
    C \big( a \star x^M \big) \ \in \R[a,x^{\pm}]^{s}
\end{equation*}
with parameters $a=(a_1,\dots,a_{m})$ and variables $x=(x_1,\dots,x_n)$. Here $C \in \R^{s \times m}$ has full row rank, the columns of $M \in \mathbb{Z}^{n \times m}$ encode the exponents of the monomials,  and $a \star x^M$ indicates that 
the $i$-th monomial  is scaled by $a_i$. A particular feature of vertical systems is that each parameter always accompanies the same monomial. 
For example, 
\begin{align*}
    \big( a_1x_1- 2 a_2x_1x_2^2 + a_3 x_2^2,\ a_2x_1x_2^2 - (a_3+a_5)x_2^2  , \ -a_4 x_1 + 3a_5 x_2^2\big)
\end{align*}
is a vertical  system with 
\[ C=\begin{pmatrix}
    1 & -2 & 1 & 0 & 0 \\ 0 & 1 & -1 & 0& -1 \\ 0 & 0 & 0& -1 & 3
\end{pmatrix}\quad \text{ and } \quad 
M=\begin{pmatrix}
    1 & 1 & 0 & 1 & 0 \\ 0 & 2 & 2 & 0 & 2
\end{pmatrix}\, . \]
Sparse (or freely parametrized) polynomial systems, in which the monomials in each entry are fixed but  the coefficient of each monomial can take any value, are a particular instance of vertical systems, see \cite{FELIU2025630}. Vertical systems appear under the name of \emph{engineered complete intersections} in \cite{Esterov:slight,Esterov:eliminate} and form an instance of a vector-valued Laurent polynomial equation as in \cite{vector-valued}. 

Motivated by their origin in   reaction network theory \cite{feinberg-book}, we include 
additional   linear entries and consider \term{augmented vertically parametrized}  systems,
\[ 
\big( C\big( a \star x^M\big) , Lx - b \big) \in \R[a,b,x^{\pm}]^{d}  \, ,
\]
where $L\in \R^{(d-s)\times n}$ has full row rank and  $b=(b_1,\dots,b_{d-s})$ are additional parameters. 

The goal of this work is to provide an efficient algorithm to decide whether an augmented vertical  system 
\term{admits multiple positive zeros} in the sense that   there 
exist  parameter values $(a,b)\in \R^{m}_{>0} \times \R^{d-s}$ and distinct $x,y \in \mathbb{R}_{>0}^n$ such that 
\[ C(a \star x^M) = C(a \star y^M) =0 \qquad  Lx = Ly = b\, .\]

Our approach builds on the \emph{higher deficiency algorithm} in the PhD thesis of Haixia Ji \cite{Ji2011UniquenessOE}, which was developed for the polynomial systems describing the steady states of a chemical reaction network. In particular, the exposition in \cite{Ji2011UniquenessOE} is tied to the language, properties, and graphical structure of reaction networks. 
The algorithm extended  previous work by the group of Feinberg \cite{Feinberg1988,advancedellison}, and  is implemented in closed-source software in the CRNT toolbox for Windows systems \cite{crnttoolbox}. 
Similar ideas and extensions have appeared in \cite{ conradiflockerzimulti,hernandez2020fundamentaldecompositionsmultistationaritypowerlaw,Hernandez_2019}.

In this work, we extend the ideas behind the higher deficiency algorithm to arbitrary augmented vertical systems, while   we also substantially simplify the exposition and clarify the  underlying algebraic formalism. In doing so, we are   able to relax some of the assumptions of \cite{Ji2011UniquenessOE} and therefore, even in the restricted setting of reaction networks, more systems can be studied with the algorithm. 
The crucial point is that, under certain hypotheses, deciding the existence of multiple positive zeros can be reduced to checking the feasibility of a linear system of equalities and inequalities. This reformulation  allows the problem to be efficiently solved with linear programming, and hence bypasses the use of all-purpose methods building on cylindrical algebraic decomposition and Gr\"obner basis computations.

To understand the main idea of the algorithm, we explain the  basic scenario. 
Assume that,  in row reduced echelon form, the matrix $C$ is of the form 
$
C = \begin{pmatrix} \id_{s} & -\oP 
\end{pmatrix}  
$
for $\oP \in \mathbb{R}^{s\times \ell}$. 
Given $\rho\in \R^{m}$, we consider the matrix   $A_{\rho}$  defined by 
\[ 
(A_{\rho})_{ij} =  \oP_{ij} 
     (e^{\rho_i} -  e^{\rho_{s+ j}}),  \qquad i \in \{1,\dots,s\}, \  j \in \{1,\dots,\ell\}\,, \]
     and 
the \textit{characteristic system} 
\[
A_{\rho}\mu = 0\,, \qquad \oP\mu \in \R_{>0}^{s}\,, \qquad \mu \in \R_{>0}^{\ell}\, .
\]
Let $\mO_L\subseteq \R^{n}$ be the union of all orthants that $\ker(L)$ intersects nontrivially. 

The first ingredient towards the algorithm is that 
the feasibility of the characteristic system  for suitable $\rho$ completely characterizes whether the system admits multiple positive zeros.

\begin{theoremalphabetic}[\cref{thm:first_characterization}]
\label[theoremalphabetic]{thm:A}
The augmented vertical  system $(C(a \star x^M), Lx - b)$ admits multiple positive zeros if and only if  the  characteristic system 
has a solution for some $\rho\in M^\top(\mO_L)$. 
\end{theoremalphabetic}

With this in place, the second ingredient is a collection of linear conditions on the entries of $\rho$ 
that are necessary for the characteristic system to have a solution. 
Formally, 
we first derive simple conditions on the signs of the entries of any 
matrix  $Q\in \R^{s\times \ell}$ that satisfies $Q\mu = 0\,$   for some $\mu \in \R_{>0}^{\ell}$ (\Cref{lem:feasible}). 
The  sign matrices   $\mathcal{S} \in \{-1,0,1\}^{s\times \ell}$ for which these conditions hold are called \textit{feasible}. 
Clearly, the sign  matrix of $A_{\rho}$ needs to be feasible if the characteristic system has a solution. 
Afterwards, we look more closely at the specific structure of the entries of $A_\rho$  and the pairs of signs of entries of $\oP$ and $A_\rho$ to derive additional linear constraints on $\rho$ for the 
characteristic system to have a solution (\cref{prop:D}). 

For any feasible sign matrix $\mS$, putting together the conditions 
$\rho \in M^\top(\mO_L)$, $\sign(A_{\rho}) = \mathcal{S}$, and the additional constraints, we 
define the \emph{feasible ground set} $\mC_{\mS} \subseteq \mathbb{R}^{m}$ (\Cref{def:feasible_ground_set}). This set is defined by linear equalities and inequalities and hence one can easily decide whether it is empty.   The union of the sets $\mC_\mS$ for all feasible matrices $\mS$ contains  all values of $\rho$ for which the characteristic system has a solution.

\begin{theoremalphabetic}[\cref{thm:mono}]  \label[theoremalphabetic]{thm:B}
 If $\mathcal{C}_{\mathcal{S}} = \varnothing$ for all feasible sign matrices $\mathcal{S}\in \{-1,0,1\}^{s\times \ell}$, then $(C(a \star x^M), Lx - b)$ does not admit multiple positive zeros. 
\end{theoremalphabetic}

The third and final ingredient is to decide whether the 
 converse of  \Cref{thm:B} holds, that is, whether $\rho\in \mC_\mS$ implies that  the characteristic system has a solution. It turns out that this holds when $\oP$  \term{induces a forest} in the sense  that the bipartite graph whose nodes correspond to the   rows and columns of $\oP$, and edges correspond to nonzero entries of $\oP$, is a forest (\cref{forest}).

\begin{theoremalphabetic}[\cref{thm:full,cor:full}]\label[theoremalphabetic]{thm:C}
Under the assumption  that  $\oP$ induces a forest, the system $( C (a\star x^M), Lx-b)$ admits multiple positive zeros if and only if $ \mathcal{C}_{\mS} \neq \varnothing$ for some  feasible matrix $\mS\in \{-1,0,1\}^{s \times  \ell}$. 
\end{theoremalphabetic}

The matrix $\oP$ will induce a forest when it  
has a particular structure of zero entries, which makes \Cref{thm:C} directly applicable only in specific situations. Our results hold in a more general setting where we  consider a submatrix $P$ of $\oP$ obtained by selecting one representative for each set of pairwise proportional rows and columns.
In this setting, we need to introduce an additional sign matrix, which we call  an  \term{orientation}, and its associated  \emph{oriented} characteristic system. 
The existence of an orientation such that the oriented characteristic system has a positive solution is equivalent to the original 
characteristic system having a solution (\Cref{thm:mainsimplification}). Hence, by \Cref{thm:A}, an augmented vertical system admits multiple positive zeros if and only if the oriented characteristic system has a solution for some orientation. \Cref{thm:B,thm:C} admit an extension to the oriented setting.  As  the new matrix $P$ may induce a forest even if that is not the case for $\oP$, our characterization of multiple positive zeros becomes applicable to a larger class of systems.

\medskip
We conclude the introduction by  highlighting three   scenarios where our results may be applicable. 
First, any polynomial system with fixed coefficients can be seen as the specialization of the vertical system obtained by including a distinct parameter in front of each distinct monomial. If this vertical system does not admit multiple positive zeros, then neither will the original system. Therefore, \Cref{thm:B} provides a method to preclude multiple positive zeros for any polynomial system.  

Second, augmented vertical systems arise naturally when studying 
 \term{reaction networks}, which are given by a collection of 
 reactions 
\[ \sum_{i = 1}^n\alpha_{ij}X_i \ce{->[$a_j$]} \sum_{i = 1}^n\beta_{ij}X_i \qquad j=1,\dots,m \]
for species $X_1,\dots,X_n$, $\alpha_{ij},\beta_{ij}\in \Z_{\geq 0}$, and $a_j>0$. 
Under the assumption  of mass-action kinetics, the   stoichiometric matrix $N=(\beta_{ij} - \alpha_{ij})$ and the reactant matrix $M  = (\alpha_{ij}) $ are used to model
 the time evolution of the concentrations $x_1,\dots,x_n$ of the species     by a system of ordinary differential equations
\[ \frac{d  x }{dt} =  N \big( a \star x^M \big)\, .\]
The trajectories of this system are constrained to the stoichiometric  classes with equations $Lx-b=0$, where $L$ is a matrix whose rows form a basis of the left kernel of $N$. 
By letting $C$ be a full row rank matrix with $\ker(N)=\ker(C)$, the positive steady states  within stoichiometric   classes are the positive zeros of the
augmented vertical   system $(C(a\star x^M),  Lx-b)$. 

Admitting at least two positive steady states is  necessary  for the network to display bistability, a property which has been linked to cell decision-making   \cite{laurentkellershohn, ozbudak}. Thus, deciding the existence of multiple positive steady states for some parameter choice   has been an active topic of research in the field; see for example \cite{feinberg-def0,craciun-feinbergI,wiuf2013power,Joshi2014ASO,signconditions,BanajiPantea2016injectivity} and the references therein.  

Third, vertical systems define the \term{critical points} of multivariate polynomials with prescribed support. Namely, given a finite set 
$\mathcal{A} \subset \mathbb{Z}^n$  and a sign function
$\varepsilon\colon \mathcal{A} \rightarrow \{-,+\}$, we may consider all polynomials with support $\mathcal{A}$ and signs of the coefficients determined by $\varepsilon$:
\[ f = \sum_{\alpha \in \mathcal{A}}\varepsilon(\alpha)\,a_{\alpha}\,x^{\alpha} \, , \qquad a \in \mathbb{R}_{>0}^{\mathcal{A}} \, .\] 
The critical positive points of $f$ are the positive zeros of the system
$(x_1\frac{df}{dx_1}  , \dots , x_n\frac{df}{dx_n})$. By letting $M$ have as columns the elements of $\mathcal{A}$ in some order, and $M_\varepsilon$ be obtained by multiplying the $i$-th column of $M$ by
$\varepsilon(\alpha_i)$, 
the critical positive zeros of $f$ are precisely the positive zeros of the  vertical  system 
\[ M_\varepsilon( a \star x^M) \, .  \]
If we are interested in the singular positive points of $f$, we may add $f$ to the system. This  gives rise to a vertical system with $n+1$ entries and $n$ variables.

Deciding whether a polynomial in this family has more than one critical positive point is thus equivalent to deciding whether a vertical system admits multiple positive zeros and can be studied with the methods of this work. 
This problem arises when studying $\mathcal{A}$-discriminants \cite{gkz1994}, the number of connected components of fewnomial hypersurfaces  \cite{feliutelek,bihan2024bounds,  telekdescartesrule}, or copositivity of polynomials and SONC decompositions \cite{Ferrer:SONC}.

The paper is structured as follows. In \Cref{section:characterization}, we present the characterization of multiple positive zeros for augmented vertical  polynomial systems using the characteristic system. \Cref{section:simplifications} introduces orientations and  the properties of the  oriented characteristic system. In \Cref{section:feasibility} we study the feasibility of oriented characteristic systems and obtain generalizations of \Cref{thm:B,thm:C}. In \Cref{sec:examples}, we illustrate our main results with several examples.
In \Cref{sec:comput}, we discuss the algorithm to decide upon multiple positive zeros derived from \Cref{thm:mono,thm:full}. 
Finally, 
in \Cref{section:properties}, we discuss, through the lens of our methods,  the connectivity of the region of parameters that lead to multiple positive zeros and the existence of multiple nondegenerate positive zeros.

\subsection *{Acknowledgments} We thank Joan Ferrer, Oskar Henriksson and Nidhi Kaihnsa for useful discussions.
This project has been funded by the European Union under the Grant Agreement number 101044561, POSALG.\footnote{Views and opinions expressed are those of the authors only and do not necessarily reflect those of the European Union or European Research Council (ERC)}

\medskip
\noindent
\textbf{Notation. } 
For two vectors $\alpha,\beta \in \mathbb{R}^{m}$,  $\alpha \star \beta$  denotes their component-wise multiplication, i.e.
$(\alpha \star \beta)_{i} = \alpha_i\beta_i$ for $i = 1,\dots,m$.
Similarly, the operations $\frac{\alpha}{\beta}, e^{\alpha}, \ln{\alpha}$ are taken component-wise.

For a  vector $v\in \R^n$, we let $\sign(v)\in \{-1,0,1\}^n$ be obtained by taking the sign entry-wise. 
For a set $V\subseteq \R^n$, we let $\sign(V)=\{\sign(v) \colon v\in V\}$. For a vector $v\in \R^n$,   $v>0$ is shorthand notation for   $v\in \R^n_{>0}$. An orthant of $\R^n$ is the set of all vectors with a given fixed sign.

For an integer $n$ we let $[n]:=\{1,\dots,n\}$. For  an  interval $[a,c]$,   $[a,c]^\circ$ denotes its relative interior.

\section{Characterization of multiple positive zeros}
\label[section]{section:characterization}

The main goal of this section is to prove \cref{thm:first_characterization}, which reformulates  the problem of deciding the existence of parameter values for which an augmented vertical system admits multiple positive zeros as a feasibility problem of an alternative system that we term the characteristic system. The key idea goes back to earlier works in the theory of chemical reaction networks, e.g. \cite{craciun-feinbergI}, and similar ideas have been used in several later works, e.g. \cite{conradiflockerzimulti, Conradi_2019, BIHAN2020107412}.

We consider an augmented vertical system
\[ F = \big( C (a\star x^M), Lx-b\big) \in \R[a,b,x^\pm]^d \]
with variables $x=(x_1,\dots,x_n)$ and parameters   $a=(a_1,\dots,a_{\bar{m}}), b=(b_1,\dots,b_{d-\bs})$, where
$M\in \Z^{n\times \bar{m}}$, $C\in \R^{\bs\times \bar{m}}$ has rank $\bs$, and $L\in \R^{(d-\bs)\times n}$ has rank $d-\bs$. We assume throughout that $C$ has no zero column.  By $F_{a^*,b^*}$ we refer to the specialization of $F$ to  given $a^*\in \R^{\bar{m}}$ and $b^*\in \R^{d-\bs}$. We introduce the bar above $s,m,\dots$ to reserve the letters without the bar to the main objects of this work to be introduced in \Cref{section:simplifications}. 

Our  aim is to decide whether \term{$F$ admits multiple positive zeros}, in the sense that there exists a choice of parameters $a\in \R^{\bar{m}}_{>0}$ and $b\in \R^{d-\bs}$ such that $F_{a,b}$ has at least two distinct zeros in $\R^n_{>0}$; equivalently
 the correspondence set  
\begin{equation}\label{eq:M}
\mathcal{M}_F  \coloneqq \big\{(a,x,y) \in \R^{\bar{m}}_{>0} \times \R^{n}_{>0} \times \R^{n}_{>0}  \colon C(a\star x^M) = C(a\star y^M) = 0, \  x \neq y, \  L(x - y) = 0\big\} \, 
\end{equation}
is nonempty.

In preparation for the main result of this section, we   introduce the following objects   and concepts associated with  $F$: 
\begin{itemize}
\item We say that  $C$ is \term{principal}  if the $(\bs\times \bs)$ submatrix of $C$  formed by the first $\bs$ columns has full rank.   
In this case, we  say that $F$ is defined by a principal matrix. 
After a suitable reordering of the columns of $C$, we may assume without 
 loss of generality that $C$ is principal in what follows. 
 
\item For $C$ principal, by letting $\bl\coloneqq \bar{m}-\bs$, there exists a unique 
matrix of the form 
\begin{equation}\label{eqP}
\widehat{P} = \begin{pmatrix} \oP \\
\id_{\bl}
\end{pmatrix} \,  \qquad \oP \in \R^{\bs \times \bl} 
\end{equation}
whose columns form a basis of $\ker(C)$. As $C$ has no zero column, neither does $\oP$. 
 The submatrix $\oP$ is called  the \term{reduced matrix} of $C$. Observe that, as the lower block of $\widehat{P}$ is the identity matrix,  for $x\in \ker(C)$ it holds $x=\widehat{P}\mu$ with $\mu= x_{\bar{s}+1,\dots,\bar{m}}$ and in particular
\begin{equation*}
x\in \ker(C)\cap \R^{\bar{m}}_{>0} \quad \Leftrightarrow \quad x= \widehat{P} \mu  \text{ for some }\mu \in \R^{\bl}_{>0}\,\text{ such that }\oP\mu>0 . 
\end{equation*}

Different orderings of the columns of $C$ yielding principal matrices may give rise to different reduced matrices. 

\item 
For a vector $\rho\in \R^{\bar{m}}$, consider the matrix 
\begin{equation}\label{eq:characteristic_matrix} 
A_{\rho} \coloneqq \begin{pmatrix} \id_{\bs} & -\oP \end{pmatrix} \diag( e^\rho)    \begin{pmatrix} \oP \\
\id_{\bl} \end{pmatrix} \quad \in \R^{\bs\times \bl}\, ,
\end{equation}
that is,
\begin{equation*}
(A_{\rho})_{ij} \coloneqq     \oP_{ij} (e^{\rho_{i}} - e^{\rho_{\bs+j}})\, ,  \qquad i = 1,\dots,\bs, \quad j = 1,\dots,\bl\, .
\end{equation*}
The \term{characteristic system} (associated with $\oP$ and $\rho$) is the system
 \begin{equation*}
A_{\rho}\mu = 0\,, \qquad \oP \mu > 0\,, \qquad \mu\in \R^{\bl}_{>0}\, . 
\end{equation*}
 
 \item We  let $\mO_L$ denote the union of all the orthants of $\R^n$ that $\ker(L)$ intersects nontrivially. 
 If $L$ is the empty matrix, then $\mO_L=\R^n\setminus \{0\}$.

\item 
 We introduce the set
\begin{align}
\overline{\mE}_F & \coloneqq \big\{(v,\delta,\mu) \in \ker(L) \times (\R^{n}\setminus \{0\}) \times \R_{>0}^{\bar\ell}  : \sign(v)=\sign(\delta),   \nonumber  \\  & \hspace{8cm}\  \oP\mu > 0, \ A_{M^{\top}(\delta)}\mu = 0 \big\}\, \label{eq:Ebar}
\end{align}
and the following two maps:
\begin{align}\label{eq:Psi}
\Psi\colon & \mM_F  \rightarrow   \R^n \times \R^n \times \R^{\bar\ell}_{>0}
\qquad (a,x,y) \mapsto  \big(x - y,  \ln x - \ln y , (a\star y^M)_{\bar{s}+1,\dots,\bar{m}}\big) \, ,  
\\[6pt]
\label{eq:Phi}
\Phi \colon & \overline{\mE}_F   \rightarrow \R^{\bar{m}}_{>0} \times \R^n_{>0}\times   \R^n_{>0} \qquad 
 (v,\delta,\mu)  \mapsto 
 \big(a(v,\delta,\mu),x(v,\delta), y(v,\delta)\big) \, ,
\end{align}
where  
\[ x(v,\delta) \coloneqq  
    \frac{v\star e^{\delta}}{e^{\delta} - \bm{1}}\,, \qquad y(v,\delta) \coloneqq  
    \frac{v}{e^{\delta} - \bm{1}}\,,  \qquad  a(v,\delta,\mu) \coloneqq \frac{\widehat{P}\mu}{y(v,\delta)^M}\, ,
\] 
with the convention that $\tfrac{0}{0}=1$ and where $\bm{1}$ is the vector of all ones. Observe that the map is well defined as the signs of $v$ and $e^{\delta}-\bm{1}$ agree and $\widehat{P}\mu>0$.

\item The  \term{characteristic set} (of $F$) is
\begin{align}
\mE_F  & \coloneqq \big\{\rho \in M^\top(\mO_L) :   \oP\mu > 0, \ A_{\rho}\mu = 0 \text{ for some } \mu\in \R_{>0}^{\bar\ell}\big\}\,  \label{eq:E}
\end{align}
There is a surjective map $\overline{\mE}_F\rightarrow \mE_F$ that sends $(v,\delta,\mu)$ to $M^\top(\delta)$. 
\end{itemize}

\begin{theorem}
\label[theorem]{thm:first_characterization}
Let $F\in \R[a,b,x^\pm]^d$ be an augmented vertical system defined by a principal matrix. 
Consider the sets $\mM_F$, $\overline{\mE}_F$ and $\mE_F$ from \eqref{eq:M}, \eqref{eq:Ebar} and \eqref{eq:E}.
\begin{itemize}
    \item The maps $\Psi,\Phi$ from \eqref{eq:Psi} and \eqref{eq:Phi} satisfy 
    \[\im \Phi \subseteq \mathcal{M}_F\,, \quad \Psi\circ\Phi=\mathrm{id}_{\overline{\mE}_F}\,,\quad \im \Psi = \overline{\mE}_F\,,\]
    and $\Phi(\Psi(a,x,y))=(a,x,y)$  unless $x_i=y_i\neq 1$ for some $i\in [n]$. 
\item $F$ admits multiple positive zeros if and only if $\mE_F\neq \varnothing$. 
\end{itemize}
\end{theorem}
\begin{proof}
Write $F =( C (a\star x^M), Lx-b)$ with $C\in \R^{\bs\times \bar{m}}$  principal.  
We first observe that with $\widehat{P},\oP$ as in \eqref{eqP} and for any $\mu\in \R^{\bl}_{>0}$, it holds
\begin{equation}\label{eq:Arho_equiv}
\begin{aligned}
e^\rho \star  (\widehat{P}\mu) \in \ker(C)= \ker \begin{pmatrix} \id_{\bs} & -\oP \end{pmatrix}\, &  \Leftrightarrow \quad \begin{pmatrix} \id_{\bs} & -\oP \end{pmatrix} \diag( e^\rho)    \begin{pmatrix} \oP \\
\id_{\bl} \end{pmatrix} \mu=0 \\
&  \Leftrightarrow \quad  A_\rho \mu=0\,  . 
\end{aligned}
\end{equation}

We verify   that $(v,\delta,\mu)\coloneqq \Psi(a,x,y)\in \overline{\mE}_F$ if $(a,x,y)\in \mM_F$. The definition of $\mM_F$ gives that $v=x-y\in \ker(L)$, $x\neq y$ so $\delta\neq 0$, and, as $\ln$ is an increasing function, $\sign(v)=\sign(\ln x - \ln y)=\sign(\delta)$.   

As $a\star y^M\in \ker(C)$ and $\mu=(a\star y^M)_{\bar{s}+1,\dots,\bar{m}}$, it holds that $a\star y^M = \widehat{P}\mu$  and hence $\bar{P}\mu>0$.  
Using now that 
\[M^\top(\delta) = M^\top \ln\big(\tfrac{x}{y}\big) = \ln\big(\tfrac{x}{y}\big)^M \quad\text{ and hence }\quad e^{M^\top(\delta)} \star y^M =  x^M\,,\]
we obtain
\[e^{M^\top(\delta)} \star \widehat{P}\mu = e^{M^\top(\delta)} \star a\star y^M = a\star x^M \] 
and since $a\star x^M\in \ker(C)$, \eqref{eq:Arho_equiv} gives that 
$ A_{M^\top(\delta)}\mu=0$. 
This shows that $\im(\Psi)\subseteq \overline{\mE}_F$.

We now show that $(a,x,y)\coloneqq\Phi(v,\delta,\mu)\in \mM_F$ if $(v,\delta,\mu)\in \overline{\mE}_F$. 
By definition  of $\Phi$,  $x-y=v\in \ker(L)$,  $x\neq y$ as $\delta\neq 0$, and 
$a \star y^M= \widehat{P}\mu\in \ker(C)$. The relation $x=y\star e^\delta$  gives
\[a\star x^M = a\star y^M \star (e^{\delta})^M  = e^{M^\top(\delta)}\star \widehat{P}\mu \in \ker(C) \]
where the last inclusion follows from  \eqref{eq:Arho_equiv} as $A_{M^\top(\delta)}\mu=0$ by hypothesis. 
Hence $\im(\Phi)\subseteq \mM_F$. 

 The composition $\Psi\circ \Phi$ is the identity map, as shown by a simple computation using that  
 \[e^{\ln x} = x\, ,\qquad  \ln \frac{v\star e^{\delta}}{e^{\delta} - \bm{1}} - \ln  \frac{v}{e^{\delta} - \bm{1}} = \delta \,,  \]
and that the columns of $\widehat{P}$  form a basis of $\ker(C)$ with the last block being the identity. Therefore, $\im \Psi = \overline{\mE}_F$ as well.
The last statement of the first bullet point follows from a computation   using that for
$(v,\delta,\mu)\coloneqq \Psi(a,x,y)$,
we have that $x_i(v,\delta) = y_i(v,\delta) = 1$ if $x_i=y_i$ by the convention $\frac{0}{0} = 1$.

Finally, $F$ admits multiple positive zeros if and only if $\mM_F\neq \varnothing$. As $\im \Psi = \overline{\mE}_F$, this is equivalent to $\overline{\mE}_F\neq \varnothing$. As there exists a surjective map from $\overline{\mE}_F$ to $\mE_F$, this concludes the proof. 
\end{proof}

\section{The oriented characteristic system}
\label[section]{section:simplifications}

By \cref{thm:first_characterization}, deciding whether $F$ admits multiple positive zeros  amounts to deciding whether $\mE_F\neq \varnothing$, that is,  
whether the characteristic system  is feasible for some $\rho$ in  $M^\top(\mO_L)$. 
Before we explore how to address this problem in \Cref{section:feasibility}, we 
take into consideration the linear dependencies among columns or rows of $\oP$ to derive a simplification of 
the characteristic system. This
 will allow us to tackle a larger class of systems   in \Cref{section:feasibility}.

For a matrix $B$, we denote its $j$-th column by $B_j$. The transpose of the $j$-th row of $B$ is then $(B^\top)_j$. 
We note that if the reduced matrix $\oP$ has two proportional nonzero rows   with a negative proportionality factor, then $\oP\mu>0$ cannot hold and $\mE_F= \varnothing$.

Given any matrix $\oP \in \R^{\bs \times \bl}$, we consider 
a \term{row  partition} $\tau=( \tau_1, \dots , \tau_{s})$ and a \term{column  partition}  $\alpha=(\alpha_1 , \dots ,\alpha_{\ell})$, 
\[
 [\bs] = \tau_1\sqcup \dots \sqcup \tau_{s}\, , \qquad  [\,\bl\,] = \alpha_1 \sqcup \dots \sqcup \alpha_{\ell} \, , \]
satisfying
\begin{align*}
 i , i' \in \tau_k \text{ for some }k &  \quad \Rightarrow \quad   (\oP^\top)_i = \gamma' (\oP^\top)_{i'} && \text{ for some } \gamma' > 0 \\
 j , j' \in \alpha_k \text{ for some }k & \quad \Rightarrow \quad   \oP_j = \gamma \oP_{j'} && \text{ for some } \gamma \neq 0 \,.
  \end{align*}
  That is, for any pair of indices in the same block of $\tau$ (resp. $\alpha$), the corresponding rows (resp. columns) of $\oP$ are proportional. 
We do not impose that indices of proportional columns or rows \emph{must} be in the same block. This allows us to consider the \term{singleton partitions}, defined by the singleton blocks 
\[ \tau_i=\{i\}\, , \quad i=1,\dots,\bs\, , \qquad\text{and}\qquad  \alpha_j=\{j\}\, ,\quad j=1,\dots,\bl\, . \] 

 Associated with partitions $\tau,\alpha$, we consider the following data:
\begin{itemize}
\item By selecting one representative column   and one representative row  in each block 
 we  obtain  maps 
 \[ c\colon  [\,\ell\,]  \rightarrow  [\, \bl \, ]\, ,  \qquad r\colon  [s] \rightarrow [\,\bs\,]  \,. \]
\item
We consider the  vectors $\gamma \in (\mathbb{R} \setminus \{0\})^{\bl}$ and  $\gamma' \in \mathbb{R}_{>0}^{\bs}$ such that
\begin{equation*}
\begin{aligned}
\oP_j &= \gamma_j \oP_{c(k)} \, ,  && \quad k \in [\ell], \ j\in \alpha_k \\
 (\oP^\top)_i & = \gamma'_i \, (\oP^\top)_{r(k)} \, , && \quad k \in [s], \ i\in \tau_k\, .  
\end{aligned}
\end{equation*}
\item We define the \term{simplified reduced matrix} $P\in \R^{s\times \ell}$ to be the submatrix of $\oP$ given by
\begin{equation}
\label{def:simpif_reduced_matrix}
P_{ik} :=  \oP_{r(i), c(k)} \,, \qquad i\in [s], \ k\in [\ell] \, . \end{equation}
\item
 We consider the blocks of $\alpha$ for which all proportionality factors are positive:  
\begin{align*}
\mathcal{U}_\alpha  &\coloneqq \{k \in [\ell] \colon \gamma_j > 0 \text{ for all }j \in \alpha_k\}\, .
\end{align*}
\end{itemize}
We refer to the triple $(\tau,r,\gamma')$ as the data of the row partition, and the triplet $(\alpha,c,\gamma)$ as the data of the  column partition. An \term{orientation compatible with $\alpha$}   is a sign matrix $\sigma \in \{-1,0,1\}^{2 \times \ell}$ such that 
\begin{equation*} 
(\sigma_{1k},\sigma_{2k}) = (1,1) \quad \text{ if }\quad k \in \mathcal{U}_\alpha\, .
\end{equation*}
Clearly, the orientation $\sigma_+$ with all entries equal to $1$  is compatible with any $\alpha$.

For any $\rho\in \R^m$ and orientation $\sigma$ compatible with $\alpha$, we consider the
matrices $P^{\sigma}$ and 
$A^\sigma_{\rho}$ 
 defined by
\begin{equation}
    \label{eq:oriented}
     (P^{\sigma})_{ik} = \sigma_{1k}P_{ik}\, , \qquad  (A^\sigma_{\rho})_{ik} \coloneqq    P_{ik} 
     (\sigma_{1k}\,e^{\rho_i} - \sigma_{2k}\,e^{\rho_{s+ k}}),  \qquad i \in [s], \  k \in [\ell]\, ,
        \end{equation}
 which give rise to the \term{oriented characteristic system}  associated with $(P,\sigma,\rho)$:      
\begin{equation}
    \label{eq:characteristic2}
A^\sigma_{\rho}\mu = 0\, , \qquad P^{\sigma}\mu > 0\,,\qquad \mu\in \R^\ell_{>0}\, .
    \end{equation}

\begin{definition}\label[definition]{def:mGsigma}
With the data above, let  $m\coloneqq s+\ell$. We define the  set 
 $\mathcal{G}^{\sigma}$ to consist of all $\rho\in \mathbb{R}^{m}$ satisfying the following for some $\bar{\rho}\in M^\top(\mathcal{O}_L)$:
\begin{enumerate}[label=(\roman*)]
 \item
     For $k\in [s]$,  $\rho_k = \bar{\rho}_i$ for all $i\in \tau_k$. 
     \item For all $k \in [\ell]$  we have
     \begin{enumerate}
     \item[(a)] If $k \in \mathcal{U}_\alpha$, then  
     \begin{equation}
     \label{eq:conditionU1}
     \rho_{s + k} \in \Big[\min_{j  \in \alpha_k}\bar{\rho}_{\bs+j},  \max_{j  \in \alpha_k}\bar{\rho}_{\bs+j} \Big]^\circ \, .
     \end{equation}
     \item[(b)] If $k \notin \mathcal{U}_\alpha$, then there exist  
\begin{equation}\label{eq:conditionU2}
\begin{aligned}
z_{+} & \in \Big[\min_{\gamma_j>0 \colon   j\in \alpha_k}\bar{\rho}_{\bs+j},  \max_{\gamma_j>0  \colon  j\in \alpha_k}\bar{\rho}_{\bs+j} \Big]^\circ,  \\ 
z_-  & \in \Big[\min_{\gamma_j<0 \colon  j\in \alpha_k}\bar{\rho}_{\bs+j},  \max_{\gamma_j<0   \colon   j\in \alpha_k}\bar{\rho}_{\bs+j} \Big]^\circ , 
\end{aligned}
\end{equation}
     such that 
\begin{equation}\label{eq:Vsigns}
\sign(e^{z_+} - e^{z_-}) = 
 \sign( \sigma_{2k}\, e^{\rho_{s+k}} -   \sigma_{1k}\, e^{z_+}  )= \sign( \sigma_{2k}\, e^{\rho_{s+k}} - \sigma_{1k} \, e^{z_-}  )\, . 
\end{equation}
\end{enumerate}
\end{enumerate}
Observe that if $\alpha_k=\{j\}$ is a singleton, then necessarily $k\in \mathcal{U}_\alpha$ and \eqref{eq:conditionU1} is simply $\rho_{s+k}=\bar{\rho}_{\bs+j}$.
\end{definition}

We define the \term{oriented characteristic set} (of $F$) as
\begin{equation*}
    \mE^\sigma_F \coloneqq \{ \rho  \in \mathcal{G}^{\sigma} :
   P^{\sigma}\mu > 0\, ,  A^\sigma_{\rho}\mu = 0 \text{ for some } \mu\in \R_{>0}^{\ell} \,\} \, . 
\end{equation*}

\begin{remark}\label[remark]{rk:singleton}
If $\alpha$ is the singleton partition, then $c$ is the identity, $\ell=\bl$, and $\mathcal{U}_\alpha=[\, \bl\,]$. Hence the only orientation compatible with $\alpha$ is $\sigma_+$ and 
\begin{multline*}
    \mathcal{G}^{\sigma_+}= \big\{\rho \in \R^m :  \text{there exists }\bar{\rho}\in M^\top(\mathcal{O}_L) \text{ such that } \\ \rho_k = \bar{\rho}_i \text{ if }i\in \tau_k\, \text{ and } \ \rho_{s+j}=\bar{\rho}_{\bs+j}\text{ for all }j\in [\ell]\big\}\, .
\end{multline*}
If $\tau$ is also the singleton partition, $r$ is the identity map, $s=\bs$,    $m=\bar{m}$, $\mathcal{G}^{\sigma_+}= M^\top(\mathcal{O}_L)$,  
 $P^{\sigma_+}=P=\bar{P}$,   $A_\rho^{\sigma_+}=A_\rho$, and it holds that 
if $P_{ij}=0$, then also $(A_\rho)_{ij}=0$. In particular $\mE_F^{\sigma_+}=\mE_F$. 
\end{remark}

 \begin{theorem}
\label[theorem]{thm:mainsimplification}
 Let $F  \in \R[a,b,x^\pm]^d$  be an augmented vertical system defined by a principal matrix and fix a choice of row and column partitions.  

 Then $\mE_F \neq \varnothing$ if and only if  $\mE^\sigma_F\neq \varnothing$ for some orientation $\sigma \in \{-1,0,1\}^{2 \times \ell}$ compatible with the column partition. 
Furthermore, if $\rho\in \mE^\sigma_F$ and $\bar{\rho}$ satisfies the conditions  in \Cref{def:mGsigma} for $\rho$, then $\bar{\rho}\in \mE_F$.  
\end{theorem}
\begin{proof}
Write $F=( C (a\star x^M), Lx-b)$ with $C\in \R^{\bs\times \bar{m}}$  principal and   reduced matrix $\oP\in \R^{\bs\times \bl}$.  Let  $(\tau,r,\gamma')$ and $(\alpha,c,\gamma)$ be the data of the  row and column partitions respectively, and consider the  associated simplified reduced matrix $P\in \R^{s\times \ell}$ from \eqref{def:simpif_reduced_matrix}.

To show the forward implication, assume $\mE_F\neq \varnothing$ and  let 
$(\bar{\rho},\bar{\mu}) \in M^\top(\mathcal{O}_L)\times \R_{>0}^{\bl}$ such that 
\begin{equation}
\label{first_assumptions}
 A_{\bar{\rho}}\, \bar{\mu} = 0\,, \qquad \oP\, \bar{\mu} > 0\, .
\end{equation}

We define an orientation $\sigma\in \{-1,0,1\}^{2\times \ell}$ as follows. For each $k \in [\ell] $, let
\[
\omega_k \coloneqq \sum_{j\in \alpha_k}\gamma_j \bar{\mu}_j\,, \qquad 
\nu_k \coloneqq \sum_{j\in \alpha_k}\gamma_j \bar{\mu}_j\, e^{\bar{\rho}_{\bs + j}}, \qquad  
\sigma(k) \coloneqq (\sign(\omega_k),\sign(\nu_k))\, .
\]
If $k \in \mathcal{U}_\alpha$, then $\omega_k,\nu_k > 0$ and hence $\sigma$ is indeed compatible with $\alpha$. 

We define  vectors $\mu\in \R^{\ell}_{>0}$ and $\rho\in \R^{m}$ entrywise for $k\in [\ell]$ and $i\in [m]$ as 
\[ \mu_{k} \coloneqq 
\begin{cases}
|\omega_k|  & \text{if }\omega_k \neq 0\, , \\
1    & \omega_k= 0\, .
\end{cases}
 \qquad 
\rho_{i} \coloneqq \begin{cases}
\bar{\rho}_{r(i)} & \text{if } i \in [s]\, , \\
\log\left(  \frac{|\nu_k |}{\mu_k} \right) & \text{if } i = s+ k\, , \  k\in [\ell], \    \nu_k \neq 0\,,  \\
0 & \text{otherwise}\, .
\end{cases}
\quad
\]
Observe that 
\begin{align}
\label{eq:psi}
\sigma_{1k} \, \mu_k= \omega_k\, , \qquad 
 \sigma_{2k}\, \mu_k \, e^{\rho_{s+k}} =  \nu_k\, ,   \qquad \text{for all }k \in [\ell]\, .
 \end{align}

Let us see that $\mu$ is a solution to the oriented characteristic system associated with $(P,\sigma,\rho)$. Using \eqref{eq:psi},
for $k \in [\ell]$ and $i\in [s]$ we have
\begin{align*}
P^\sigma_{ik}\mu_k  &=  \sigma_{1k} P_{ik} \, \mu_k  = P_{ik} \, \omega_k =   \sum_{j\in \alpha_k} (\gamma_j \oP_{r(i),c(k)})  \bar{\mu}_j  = 
 \sum_{j\in \alpha_k}\oP_{r(i),j}\bar{\mu}_j \,.   \\
 (A_{\rho}^\sigma)_{ik}\, \mu_k &= 
P_{ik}(\sigma_{1k}\, e^{\rho_i} - \sigma_{2k}\, e^{\rho_{s + k}})\mu_k 
=P_{ik} (e^{\rho_i}\, \omega_k  - \nu_k) \\
&= \sum_{j\in \alpha_k} \gamma_j \oP_{r(i), c(k)}  (e^{\bar{\rho}_{r(i)}}- e^{\bar{\rho}_{\bs+j}}) \bar{\mu}_j  
=  \sum_{j\in \alpha_k} \oP_{r(i),j} (e^{\bar{\rho}_{r(i)}} - e^{\bar{\rho}_{\bs + j}} )\bar{\mu}_j  =  \sum_{j\in \alpha_k}   (A_{\bar{\rho}})_{r(i),j} \bar{\mu}_j \, . 
 \end{align*} 
Summing over all $k\in [\ell]$ and using \eqref{first_assumptions} it follows that  
\[ A_{\rho}^\sigma  \, \mu = 0\, ,\qquad  P^{\sigma} \mu  > 0\,.\] 

We verify now that $\rho\in \mathcal{G}^{\sigma}$ by showing that $\rho$ satisfies  conditions (i)-(ii) in \Cref{def:mGsigma} with the given $\bar{\rho}$. For $k\in [s]$ and $i_1,i_2 \in \tau_k$,  \eqref{first_assumptions} gives that
\begin{equation*}
0= \sum_{j=1}^{\bl} \left(\tfrac{1}{\gamma'_{i_1}}(A_{\bar{\rho}})_{i_1,j}  -\tfrac{1}{\gamma'_{i_2}}(A_{\bar{\rho}})_{i_2,j}\right) \bar{\mu}_j  = \left(\sum_{j=1}^{\bl}  \oP_{r(k),j}\bar{\mu}_j \right) (e^{\bar{\rho}_{i_1}} - e^{\bar{\rho}_{i_2}})\, .
\end{equation*}
As $\oP\, \bar{\mu}>0$, this gives that $\bar{\rho}_{i_1}=\bar{\rho}_{i_2}$ for all $i_1,i_2\in \tau_k$. 
Hence, condition (i) in \Cref{def:mGsigma} holds. 
For condition (ii), if $k \in \mathcal{U}_\alpha$, then 
$\omega_k, \nu_k >  0$ and 
\[e^{\rho_{s+k}} = \frac{\nu_k}{\omega_k} = \frac{ \sum_{j\in \alpha_k}\gamma_j\bar{\mu}_j\, e^{\bar{\rho}_{\bs + j}}}{ \sum_{j\in \alpha_k}\gamma_j\bar{\mu}_j} \in \Big[\min_{j\in \alpha_k} e^{\bar{\rho}_{\bs+j}},  \max_{j\in \alpha_k} e^{\bar{\rho}_{\bs+j}} \Big]^\circ \, .  \] 
As  the exponential function is strictly increasing, condition (ii)(a) in \eqref{eq:conditionU1} holds. 

If $k\in [\ell]\setminus \mathcal{U}_\alpha$, then let 
\[\omega_+ \coloneqq \sum_{\gamma_j > 0 \colon  j\in \alpha_k}\gamma_j\bar{\mu}_j>0\,, \qquad \omega_-\coloneqq\sum_{\gamma_j < 0 \colon  j\in \alpha_k}\gamma_j\bar{\mu}_j<0\, , \]
and $z_+$ and $z_-$ be defined such that 
\[
e^{z_+} = \frac{\sum_{\gamma_j > 0 \colon j\in \alpha_k} \gamma_j\, \bar{\mu}_j\, e^{\bar{\rho}_{\bs+j}}}{\omega_+}>0\, , 
 \qquad e^{z_-} = \frac{\sum_{\gamma_j < 0 \colon j\in \alpha_k}\gamma_j\, \bar{\mu}_j\, e^{\bar{\rho}_{\bs+j}}}{\omega_-}>0 \, . 
 \]
By construction,   $z_+,z_-$ belong to the intervals indicated in condition (ii)(b) in  \eqref{eq:conditionU2}.  
Additionally, using that  $\omega_+= \omega_k-\omega_-$, $ \nu_k=\omega_+e^{z_+} + \omega_- e^{z_-}$, and  \eqref{eq:psi}, we obtain that 
\begin{align*}
\omega_+ (e^{z_+} - e^{z_-})  & = \nu_k - \omega_k e^{z_-}    =  \mu_k \, ( \sigma_{2k}\,  e^{\rho_{s+k}}  -  \sigma_{1k} \, e^{z_-})  \, .
\end{align*}
As $\omega_+>0$ and $\mu_k>0$, we have shown that $\sign(e^{z_+}-e^{z_-}) = \sign( \sigma_{2k}e^{\rho_{s+k}} - \sigma_{1k}  e^{z_-})$. 
The other   equality of signs in \eqref{eq:Vsigns} follows analogously by studying $\omega_- (e^{z_+} - e^{z_-})$. This shows that $\rho \in \mE^\sigma_F$ and hence $\mE^\sigma_F\neq \varnothing$.

\medskip
We now show the reverse implication and the last claim of the statement. Assume there exists an orientation $\sigma$
such that $\mE_F^\sigma\neq \varnothing$. Let $(\rho,\mu) \in \mathcal{G}^{\sigma}\times \mathbb{R}^{\ell}_{>0}$ such that
\begin{equation}\label{eq:intermediate1}
A_{\rho}^{\sigma}\mu= 0\,, \qquad P^{\sigma}\mu > 0 \, .
\end{equation}
Let $\bar{\rho}\in M^\top(\mathcal{O}_L)$ satisfy  the conditions of the definition of $\mathcal{G}^{\sigma}$ in  \Cref{def:mGsigma} for this $\rho$. 
We have in particular,  $\bar{\rho}_i=\rho_{k}$ if $i\in \tau_k$ by condition (i). 

We fix $k\in [\ell]$. Using that $\rho\in \mathcal{G}^{\sigma}$, we consider:
\begin{itemize}
\item 
 If $k \notin \mathcal{U}_\alpha$ (hence in particular $\alpha_k$ is not a singleton), then there exist $z_+,z_-\in \R$ and   $\beta_j \in (0,1)$ for all $j\in \alpha_k$
     such that  
    \begin{align*}
e^{z_+} &= \sum_{\gamma_j > 0\colon j\in \alpha_k}\beta_j e^{\bar{\rho}_{\bs+j}}\, ,  & \sum_{ \gamma_j > 0\colon j\in \alpha_k}\beta_j  &= 1 \, , \\ 
e^{z_-} &= \sum_{\gamma_j < 0 \colon j\in \alpha_k}\beta_je^{\bar{\rho}_{\bs+j}}\, ,   &  \sum_{\gamma_j < 0\colon j\in \alpha_k}\beta_j &= 1 \, . 
\end{align*}
Condition \eqref{eq:Vsigns} guarantees that  
there exist $\omega_+,\omega_->0$ satisfying the linear system
\begin{equation}\label{eq:omega}
\omega_+-\omega_-=\sigma_{1k} \,  \quad \text{   and } \quad 
e^{z_+}\omega_+- e^{z_-}  \omega_-  = \sigma_{2k}e^{\rho_{s+k}}\, . 
\end{equation}

\item If 
$k \in \mathcal{U}_\alpha$ and $\alpha_k=\{j\}$ is a singleton,  let 
$\beta_j=1$, $z_+\coloneqq\bar{\rho}_{\bs+j}$ and $z_-\coloneqq1$. Otherwise, for all $j\in \alpha_k$  there exists $\beta_j \in (0,1)$  such that
\[  e^{\rho_{s + k}} = \sum_{j\in \alpha_k} \beta_j  e^{\bar{\rho}_{\bs+j}} ,\qquad    \sum_{j\in \alpha_k} \beta_j = 1\, . \]
Define    $z_+\coloneqq\rho_{s + k}$ and $z_-\coloneqq1$, and note that by letting $\omega_+\coloneqq 1$, $\omega_-\coloneqq 0$, then \eqref{eq:omega} also holds. 
 \end{itemize}

   With this in place, define, for $j\in \alpha_k$,
   \begin{equation}\label{eq:mubar}
\bar{\mu}_j \coloneqq \begin{cases}
\frac{\omega_+ \beta_j}{\gamma_j} \mu_k & \text{if } \gamma_j > 0 \, ,  \\ 
-\frac{\omega_- \beta_j}{\gamma_j}  \mu_k &  \text{if } \gamma_j < 0 \, .
\end{cases} \end{equation}
Observe that $\bar{\mu}_j>0$ for all $j$. 
Using \eqref{eq:omega}, \eqref{eq:mubar} and the definition of $P$ and $\gamma$, it now holds for $i \in [\bs]$ with $i\in \tau_u$  that
\begin{align*}
\sum_{j\in \alpha_k}\oP_{ij}\bar{\mu}_j   & = 
   \gamma_i' P_{uk} \left( \sum_{\gamma_j > 0 \colon j\in \alpha_k}\beta_j\, \omega_+  - \sum_{\gamma_j < 0\colon j\in \alpha_k}\beta_j \, \omega_- \,\right) \mu_k \\
& = \gamma_i'    P_{uk}\, (\omega_+ -  \omega_- )  \mu_k 
 = \gamma_i'   \sigma_{1k}P_{uk}\, \mu_{k} 
 = \gamma_i'     P^\sigma_{uk}\, \mu_k  \, . 
\end{align*}
Similarly, using that $\bar{\rho}_i=\rho_{u}$,  we have
\begin{align*}
\sum_{j\in \alpha_k}(A_{\bar{\rho}})_{ij} \bar{\mu}_j  &= 
\gamma_i' P_{uk} \left(
\sum_{\gamma_j > 0 \colon j\in \alpha_k}\beta_j(e^{\bar{\rho}_i} - e^{\bar{\rho}_{\bs+j}})\omega_+\, 
- \sum_{\gamma_j < 0 \colon j\in \alpha_k } \beta_j (e^{\bar{\rho}_i} - e^{\bar{\rho}_{\bs+j}})\omega_-\,  \right)  \mu_k \\
&=\gamma_i'P_{uk} \Big( \sigma_{1k}\, e^{\rho_{u}}  + \big( e^{z_-}  \omega_-  - e^{z_+}\omega_+
\big) \Big) \mu_k\,   \\
&= \gamma_i' P_{uk} \big(\sigma_{1k} \, e^{\rho_{u}}  - \sigma_{2k}e^{\rho_{s+k}}\big) \mu_k   =  \gamma_i'   (A_{\rho}^{\sigma})_{uk}\,  \mu_k  \, . 
\end{align*}

As these equalities hold for all $k$, summing over all $k$, using $\gamma_i'>0$, and \eqref{eq:intermediate1}, it follows that $\bar{\rho} \in \mE_F$, which concludes the proof.
\end{proof}

\section{Feasibility of the oriented characteristic system}
\label[section]{section:feasibility}

By \cref{thm:first_characterization,thm:mainsimplification}, to decide whether $F$ admits multiple positive zeros  it is enough to decide 
whether $\mE^\sigma_F\neq \varnothing$ for some orientation $\sigma$. 
In this section, we first show that $\mE^\sigma_F$ is always contained in a union of relative interiors of polyhedral cones, and hence, if the union is empty, the system does not admit multiple positive zeros. This criterion applies to any augmented vertical system.
Second, we show that this union is exactly equal to $\mE^\sigma_F$ when the simplified reduced matrix $P$ \emph{induces a forest} (\cref{forest}), which requires that $P$ is a matrix with a specific structure of zero entries. In this case, the existence of multiple positive zeros is completely characterized by one of the relative interiors being nonempty. 
 
\subsection{Necessary conditions for multiple positive zeros}
We start by exploring simple conditions on the signs of the entries of  $A_{\rho}^{\sigma}$ that preclude the existence of solutions to the oriented characteristic system \eqref{eq:characteristic2}, hence $\mE_F^\sigma=\varnothing$.

\begin{definition}\label[definition]{def:feasible}
    We say that a sign matrix $\mathcal{S}\in \{-1,0,1\}^{s \times \ell}$ is \term{feasible}  if   for all $i \in [s]$,  the $i$-th row is either zero or has at least one entry equal to $1$ and one equal to $-1$.
\end{definition}

\begin{lemma} 
\label[lemma]{lem:feasible}
Consider a sign matrix $\mathcal{S}\in \{-1,0,1\}^{s \times \ell}$ and $\mu\in \R^\ell_{>0}$.
There exists  a matrix $Q \in \mathbb{R}^{s \times \ell}$  such that 
\[\sign(Q) = \mathcal{S}\, \quad \text{and}\quad  Q\mu = 0 \]
if and only if $\mathcal{S}$ is feasible.
\end{lemma}
\begin{proof}
Let $Q \in \mathbb{R}^{s \times \ell}$  such that $\sign(Q) = \mathcal{S}$ and assume that $\mathcal{S}$ is not feasible. Then 
 there exists $i\in [s]$ such that 
 the $i$-th row of $\mS$ is nonzero and all nonzero entries are equal. 
This gives that either 
$\sum_{j=1}^\ell Q_{ij}\mu_j < 0$ or $\sum_{j=1}^\ell Q_{ij}\mu_j > 0$, a contradiction.  

To show the reverse implication, assume that $\mathcal{S}$ is  feasible. If the $i$-th row of $\mS$ is zero, let $Q_{ij}=0$ for all $j\in [\ell]$.  
Otherwise, as $\mS$ is feasible, the $i$-th row of $\mS$ contains both signs $1$ and $-1$. As $\mu>0$, the entries of the $i$-th row of $Q$ can be chosen to be orthogonal to $\mu$ and satisfy $\sign(Q_{ij}) = \mathcal{S}_{ij}$ for  $j \in [\ell]$. 
\end{proof}

By 
\cref{lem:feasible}, if $\sign(A_{\rho}^{\sigma})$ is not  feasible, then $\mE^\sigma_F=\varnothing$. We explore next 
additional   conditions on $\rho$ that guarantee that \eqref{eq:characteristic2} has no solutions   and that depend on the pairs of signs of the entries of $P^\sigma$ and $A_{\rho}^\sigma$. This leads to the introduction of the following sets.

\begin{definition}
\label[definition]{def:Subsets}
Let $P \in \R^{s\times \ell}$ and consider sign matrices $\mathcal{S}\in \{-1,0,1\}^{s \times \ell}$ and $\sigma \in \{-1,0,1\}^{2 \times \ell}$. 
For $i \in [s]$, we let 
\begin{align*}
  \Lambda_i^{++} & \coloneqq \{j \in [\ell]: P^\sigma_{ij} > 0,  \mathcal{S}_{ij} >0 \}\, &
\Lambda _i^{+-} &\coloneqq \{j \in [\ell]: P^\sigma_{ij} > 0 , \mathcal{S}_{ij}  <0  \}\,  
\\
 \Lambda_i^{-+} & \coloneqq \{j \in [\ell] : P^\sigma_{ij} < 0 , \mathcal{S}_{ij} >0 \}\, &  \Lambda_i^{--} &\coloneqq \{j \in [\ell]: P^\sigma_{ij} < 0 , \mathcal{S}_{ij} <0\}\, 
\end{align*}
 and consider also the sets
 \begin{align*}
\Lambda^{*+}_i  & \coloneqq \{j \in [\ell]:   \mathcal{S}_{ij} >0 \}\,,
&  
 \Lambda^{*-}_i  & \coloneqq \{j \in [\ell]:   \mathcal{S}_{ij} <0 \}\,, 
 & \Lambda_i  & \coloneqq 
 \{j \in [\ell]:   \mathcal{S}_{ij}\neq  0  \} \, . 
 \end{align*}
We refer to  the collection of these sets  as the  \term{$\bm{\Lambda}$-sets} of $(P,\sigma,\mS)$. 
\end{definition}

Observe that  $\Lambda_i=\varnothing$ if and only if the $i$-th row of $\mS$ is zero and that if $\mS$ is feasible, then   $\Lambda_i^{*+}   \neq \varnothing$ and $\Lambda_i^{*-}   \neq \varnothing$ whenever  $\Lambda_{i} \neq \varnothing$.

 \begin{definition}\label[definition]{def:ID}
Given  $P\in \R^{s\times \ell}$ and sign matrices $\mathcal{S}\in \{-1,0,1\}^{s \times \ell}$ and $\sigma \in \{-1,0,1\}^{2 \times \ell}$, 
consider the $\Lambda$-sets of $(P,\sigma,\mS)$ from \Cref{def:Subsets}.
Associated with these data we define the disjoint sets
 \begin{equation*}
 \begin{aligned}
 I^+ & \coloneqq \Big\{i \in [s]  : \Lambda_i^{*-} = \Lambda_i^{--}\neq \varnothing, \; \Lambda_i^{++} \neq \varnothing \Big\}\,   \\
I^- & \coloneqq \Big\{i \in [s]  : \Lambda_i^{*+} = \Lambda_i^{-+} \neq \varnothing, \;  \Lambda_i^{+-} \neq \varnothing\Big\}\, 
\end{aligned}
\end{equation*}
and for $i\in I^+\cup I^-$ the cone
\begin{equation*}
 \begin{aligned}
 \mD_i & \coloneqq \begin{cases} 
  \Big\{\rho \in \R^m   :  \max_{j \in \Lambda_i^{++}}\frac{\sigma_{2j}}{\sigma_{1j}}\, e^{\rho_{s + j}} \leq \min_{j \in \Lambda_i^{--}}\frac{\sigma_{2j}}{\sigma_{1j}}\, e^{\rho_{s + j}}  \Big\}\,  & \quad \text{if }i\in I^+  \\[8pt]
  \Big\{\rho \in \R^m   : \max_{j \in \Lambda_i^{-+}}\frac{\sigma_{2j}}{\sigma_{1j}}\, e^{\rho_{s + j}} \leq \min_{j \in \Lambda_i^{+-}}\frac{\sigma_{2j}}{\sigma_{1j}}\, e^{\rho_{s + j}} \Big\}   & \quad\text{if }i\in I^-.
 \end{cases}
\end{aligned}
\end{equation*}
\end{definition}

 Observe that the restrictions on $\rho$ for belonging to $\mD_i$  involve only  the last $\ell$ entries of $\rho$ and that $\sigma_{1j}\neq 0$ for all  $j$'s considered in the definition of $\mD_i$.  The requirement $\rho \in \mathcal{D}_i$ corresponds to the conditions   in \cite[Section 2.10]{Ji2011UniquenessOE} relying on the notion of \emph{segregated} index sets.
 
\begin{proposition}
\label[proposition]{prop:D}
Let $P\in \R^{s\times \ell}$ and $\sigma \in \{-1,0,1\}^{2 \times \ell}$. Given $\rho\in \R^m$, consider the $\Lambda$-sets of $(P,\sigma,\sign(A_{\rho}^{\sigma}))$ from \Cref{def:Subsets} and the associated sets $I^+,I^-,\{\mD_i\}_{i\in I^+\cup I^-}$ from \Cref{def:ID}. 

If   $\mu\in \R^\ell_{>0}$ satisfies $P^\sigma \mu>0$ and for some 
    $i\in I^+\cup I^-$ it holds
  \begin{equation}\label{eq:lambda_extra}
  \rho \in \mD_i \qquad \text{and}\qquad \sum_{j\colon 
 \sign(A_{\rho}^{\sigma})_{ij}=0} P_{ij}^{\sigma}\mu_j \leq 0\,,
  \end{equation}
then $\mu$ is not a  solution to the oriented characteristic system associated with $(P,\sigma,\rho)$.
\end{proposition}
\begin{proof}
We consider two cases.
If $i \in I^+$, then by letting $\theta \coloneqq \max_{j \in \Lambda_i^{++}}\frac{\sigma_{2j}}{\sigma_{1j}}\, e^{\rho_{s + j}}$ we have 
\[ \frac{\sigma_{2j_1}}{\sigma_{1j_1}}\, e^{\rho_{s + j_1}}\leq  \theta  \leq \frac{\sigma_{2j_2}}{\sigma_{1j_2}}\,  e^{\rho_{s + j_2}} \qquad \text{for all }j_1\in \Lambda_{i}^{++},\  j_2\in \Lambda_{i}^{--}\, .\] 
Note that for any $j\in \Lambda_{i}^{++}\cup \Lambda_{i}^{--}$, the signs of $P^\sigma_{ij}=\sigma_{1j}P_{ij}$ and of
\[ (A_\rho^\sigma)_{ij}=P_{ij}(\sigma_{1j}e^{\rho_i} - \sigma_{2j}e^{\rho_{s + j}})= P_{ij}^\sigma(e^{\rho_i} - \tfrac{\sigma_{2j}}{\sigma_{1j}}e^{\rho_{s + j}}  )\] are nonzero and agree. Therefore, 
 $e^{\rho_i} - \frac{\sigma_{2j}}{\sigma_{1j}}\, e^{\rho_{s + j}}>0$ and in particular $e^{\rho_i} - \theta>0$. It follows that 
$P^{\sigma}_{ij}(e^{\rho_i} - \frac{\sigma_{2j}}{\sigma_{1j}}\, e^{\rho_{s + j}}) \geq P^{\sigma}_{ij}(e^{\rho_i} - \theta)$ for any such $j$.  

Using that $\Lambda_i= \Lambda_i^{*+}    \cup\Lambda_i^{--}$,  that $(A_{\rho}^{\sigma})_{ij}>0$  and $P^{\sigma}_{ij}\leq 0$ if $j\in \Lambda_i \setminus (\Lambda_i^{++}\cup \Lambda_i^{--})$,  
and \eqref{eq:lambda_extra}, we have
\begin{multline*}
 (A_\rho^{\sigma} \mu)_i = 
 \sum_{j \in \Lambda_{i}} (A_\rho^{\sigma})_{ij} \mu_j \geq   \sum_{j \in \Lambda_{i}^{++} \cup  \Lambda_{i}^{--}} (A_\rho^{\sigma})_{ij} \mu_j 
 = \sum_{j \in \Lambda_{i}^{++} \cup  \Lambda_{i}^{--}}P^{\sigma}_{ij}(e^{\rho_i} - \tfrac{\sigma_{2j}}{\sigma_{1j}}e^{\rho_{s + j}})\mu_j   
    \\
     \geq   \Big( \sum_{j \in \Lambda_{i}^{++}\cup  \Lambda_{i}^{--}} P^\sigma_{ij}\mu_j \Big)(e^{\rho_i} - \theta)    \geq (P^{\sigma}\mu)_i (e^{\rho_i} - \theta)  >0 \, .
  \end{multline*}
Hence, $A_\rho^{\sigma} \mu \neq 0$. 

If $i \in I^{-}$, we proceed similarly to above.  By letting $\theta \coloneqq \min_{j \in \Lambda_i^{+-}}\frac{\sigma_{2j}}{\sigma_{1j}}\, e^{\rho_{s + j}}$ we have
\[
\frac{\sigma_{2j_1}}{\sigma_{1j_1}}\, e^{\rho_{s+j_1}}\leq  \theta  \leq \frac{\sigma_{2j_2}}{\sigma_{1j_2}}\, e^{\rho_{s+j_2}}
\qquad \text{for all }j_1\in \Lambda_{i}^{-+}, \ j_2\in \Lambda_{i}^{+-} \, . \]
For $j\in \Lambda_{i}^{-+}\cup \Lambda_{i}^{+-}$,  the signs of $P^\sigma_{ij}$ and $(A^\sigma_{\rho})_{ij}=P_{ij}(\sigma_{1j}e^{\rho_i} - \sigma_{2j}e^{\rho_{s + j}})=P_{ij}^\sigma(e^{\rho_i} - \tfrac{\sigma_{2j}}{\sigma_{1j}}e^{\rho_{s + j}})$ are nonzero but opposite. Thus, $ \frac{\sigma_{2j}}{\sigma_{1j}}\, e^{\rho_{s + j}} - e^{\rho_i}  >0$, $\theta-e^{\rho_i}>0$ and 
$ P^\sigma_{ij}(\theta - e^{\rho_i}  ) \leq P^\sigma_{ij}\big(\frac{\sigma_{2j}}{\sigma_{1j}}\, e^{\rho_{s + j}}- e^{\rho_i}\big)$. 

We use now that $\Lambda_i= \Lambda_i^{-+}  \cup \Lambda_i^{*-}$, that $(A_{\rho}^\sigma)_{ij}<0$
and $P^\sigma_{ij}\leq 0$ if $j\in \Lambda_i\setminus (\Lambda_i^{-+} \cup \Lambda_i^{+-})$,  and \eqref{eq:lambda_extra}, to obtain
\begin{multline*}
 0 <(\theta - e^{\rho_i}  ) (P^{\sigma}\mu)_i \leq   \sum_{j \in \Lambda_{i}^{+-} \cup  \Lambda_{i}^{-+}} P^{\sigma}_{ij} (\theta - e^{\rho_i}  )  \mu_j  
  \leq    \sum_{j \in \Lambda_{i}^{+-} \cup  \Lambda_{i}^{-+}} P^{\sigma}_{ij} (\tfrac{\sigma_{2j}}{\sigma_{1j}}e^{\rho_{s+j} }- e^{\rho_i}  )  \mu_j     
     \leq   -(A_\rho^\sigma \mu)_i\, .
  \end{multline*}
  Hence, $A_\rho^\sigma \mu\neq 0$. 
  
We have shown that $A_\rho^\sigma \mu\neq 0$ whenever  $P^\sigma \mu>0$ and $\mu>0$, and hence $\mu$ is not a solution to the oriented characteristic system associated with $(P,\sigma,\rho)$.
\end{proof}
 
\begin{definition}\label[definition]{def:feasible_ground_set}
Given $P \in \R^{s\times \ell}$,   $\sigma\in \{-1,0,1\}^{2\times \ell}$, and   $\mathcal{S} \in \{-1,0,1\}^{s \times \ell}$, recall the set $\mG^\sigma$ from  \Cref{def:mGsigma}, the associated sets $I^+,I^-,\{\mD_i\}_{i\in I^+\cup I^-}$ from \Cref{def:ID}, and the matrices  $P^\sigma$ and $A_\rho^\sigma$ from \eqref{eq:oriented} for $\rho\in \R^{s+\ell}$. 

The associated \term{feasible ground set} $ \mathcal{C}_{\mS}^\sigma$ is defined to be
\begin{equation*} 
    \mathcal{C}_{\mS}^\sigma \coloneqq\bigcup_{J \subseteq I^+\cup I^- \colon \Gamma_J\neq \varnothing} \Big\{\rho \in \mathcal{G}^{\sigma} : \sign(A_{\rho}^\sigma) = \mathcal{S},\  J=\{i \in I^+\cup I^-: \rho\in \mathcal{D}_i\} \Big\} \, ,    
\end{equation*}
where 
\[ \Gamma_J \coloneqq \big\{\mu \in \mathbb{R}^{\ell}_{>0}  : P^{\sigma}\mu > 0, \; \sum_{j\colon 
 \mathcal{S}_{ij}=0} P_{ij}^{\sigma}\mu_j > 0\ \text{ for all }i\in J \big\}\,.\]
\end{definition}

 \begin{theorem}[Preclusion of multiple positive zeros]\label[theorem]{thm:mono}
 Let $F  \in \R[a,b,x^\pm]^d$ be an augmented vertical system defined by a principal matrix. Consider a choice of row and column partitions and let  $P\in \R^{s\times \ell}$ be  the  associated simplified reduced matrix  from \eqref{def:simpif_reduced_matrix}. 

\begin{itemize}
    \item 
For every orientation $\sigma \in \{-1,0,1\}^{2 \times \ell}$   compatible with the column partition it holds  
\[ \mE^\sigma_F\  \subseteq \ \bigcup_{\mS} \mathcal{C}_{\mathcal{S}}^\sigma\, , \]
where the union is over all feasible matrices $\mathcal{S}\in \{-1,0,1\}^{s\times \ell}$.

\item If $\mathcal{C}_{\mathcal{S}}^\sigma = \varnothing$ for all orientations $\sigma \in \{-1,0,1\}^{2 \times \ell}$   compatible with the column partition and  all  feasible matrices $\mathcal{S}\in \{-1,0,1\}^{s\times \ell}$, then $F$ does not admit multiple positive zeros for any choice of parameter values.
 \end{itemize}
\end{theorem}
\begin{proof}
If $\rho\in \mE^\sigma_F$, then $\rho\in \mG^\sigma$ and the oriented characteristic system has a solution $\mu$. This requires $\sign(A_{\rho}^\sigma)$ to be feasible by \Cref{lem:feasible}, $P^\sigma\mu>0$, and 
$\mu\in \Gamma_J$ with $J=\{i \in I^+\cup I^-: \rho\in \mathcal{D}_i\}$ by \Cref{prop:D}. Hence  $\Gamma_J\neq \varnothing$ and $\rho\in \mC_{\sign(A_{\rho}^\sigma)}^\sigma$. 

The second statement follows from the first, using  \Cref{thm:first_characterization,thm:mainsimplification}.
 \end{proof}
 
\begin{remark}
\cref{thm:mono} is analogous to \cite[Proposition 2.10.14]{Ji2011UniquenessOE}, which is one of the key results in the higher deficiency algorithm.
\end{remark}

 \subsection{Characterization of multiple positive zeros}
 \cref{thm:mono} gives a necessary condition for the existence of multiple positive zeros. After imposing some conditions on the simplified reduced matrix $P$,   we show below in \Cref{thm:multiSimplified} that the converse to \cref{thm:mono} also holds. 
 The following   technical lemma will be key in the proof of \Cref{thm:multiSimplified}.

\begin{lemma}\label[lemma]{mainlemma}
Let $P\in \R^{s\times \ell}$, $\sigma \in \{-1,0,1\}^{2 \times \ell}$,  $\mathcal{S} \in \{-1,0,1\}^{s \times \ell}$  feasible, and consider the associated $\Lambda$-sets from \Cref{def:Subsets}. 
Assume   $P^\sigma\mu^*>0$ for some $\mu^*\in \R^\ell_{>0}$ and let 
$\rho \in \mathcal{C}^\sigma_{\mS}$. 
Then, there exists
 $Q \in \mathbb{R}^{s \times \ell}$ with $\sign(Q) = \mathcal{S}$ and $\mu\in \R^\ell_{>0}$ such that  
       \begin{equation}\label{eq2MainLemma} 
       Q \mu = 0\,,\quad P^\sigma \mu>0\,, \text{ and }
         \sum_{j\colon P^\sigma_{ij}\mS_{ij}\neq 0\ }\frac{Q_{ij} \, \mu_j}{e^{\rho_i} - \frac{\sigma_{2j}}{\sigma_{1j}}\, e^{\rho_{s + j}}}  + \sum_{j\colon \mS_{ij}=0}P^{\sigma}_{ij}  \mu_j  > 0  \, \text{ for all }i \in [s]\,.
      \end{equation}
\end{lemma}
\begin{proof}
As $\rho \in \mathcal{C}_{\mS}^\sigma$, $\sign(A_{\rho}^\sigma) = \mathcal{S}$ and 
there exists $\mu\in \Gamma_J$ where
$J\coloneqq \{i \in I^+\cup I^-: \rho\in \mathcal{D}_i\}$. 
In particular $P^\sigma \mu>0$. 
As $\mS$ is feasible, by \cref{lem:feasible}, there exists $\widetilde{Q}\in \mathbb{R}^{s \times \ell} $ with $\sign(\wQ) = \mathcal{S}$  such that  $\wQ \mu = 0$. 
We construct now a parametric matrix  $Q(\omega)$ for $\omega = (\omega_1,\dots,\omega_s) \in \mathbb{R}_{>0}^s$ such that, for every $i\in [s]$,  $\omega_i$ only modifies the $i$-th row and it holds that 
\begin{equation}\label{eq:row_mod}
\sign(\wQ_{ij})=\sign(Q(\omega)_{ij}) \quad\text{for all } j\in [\ell]\, ,   \qquad (Q(\omega)\mu)_i=0 \, , 
\end{equation} 
and \eqref{eq2MainLemma} holds after choosing  $\omega_i$ either small or large enough.

If $i \in [s]$ is  such that 
\begin{equation}\label{eq:omega_small}
    i \in J\,  \quad\text{or}\quad  \Lambda_i^{++} \cup \Lambda_i^{+-} = \varnothing\, ,
\end{equation}  
then $\sum_{j\colon \mS_{ij}=0}P^{\sigma}_{ij}  \mu_j>0\,$. 
Indeed, this holds by the choice of $\mu$ if $i \in J$,  and if $\Lambda_i^{++} \cup \Lambda_i^{+-} = \varnothing$, then we have
   \[0 < (P^{\sigma}\mu)_i = \sum_{j \in \Lambda_{i}^{-+} \cup \Lambda_{i}^{--}}P^{\sigma}_{ij}\mu_j + \sum_{j\colon \mS_{ij}=0}P_{ij}^{\sigma}\mu_j \leq \sum_{j\colon \mS_{ij}=0}P_{ij}^{\sigma}\mu_j\, .  \]
By letting
\begin{equation*} 
Q(\omega)_{ij} \coloneqq \omega_i\, \wQ_{ij}, \qquad \text{if }j \in \Lambda_{i} \, ,
\end{equation*}
and zero otherwise, \eqref{eq:row_mod} clearly holds for $\omega_i>0$  and the left hand side of \eqref{eq2MainLemma} becomes
\begin{equation} \label{eq:firstModificationCheck}
\omega_i \sum_{j \colon P^\sigma_{ij}\mS_{ij}\neq 0\ }\frac{\wQ_{ij} \, \mu_j}{e^{\rho_i} - \frac{\sigma_{2j}}{\sigma_{1j}}\, e^{\rho_{s + j}}} + \sum_{j \colon \mS_{ij}=0}P^{\sigma}_{ij}  \mu_j\, . 
\end{equation}
 As the constant term is positive, we can choose $\omega_i > 0$ small enough such that \eqref{eq:firstModificationCheck} is positive.

 \smallskip
If $i\in [s]$ does not satisfy \eqref{eq:omega_small}, then 
we  pick two indices $j_1,j_2$ such that $\mS_{ij_1}>0$ and $\mS_{ij_2}<0$ as follows:
\begin{enumerate}
\item If $\Lambda_i^{*-} \neq \Lambda_i^{--}$ and $\Lambda_i^{*+} \neq \Lambda_i^{-+}$, then as \eqref{eq:omega_small} does not hold, we can pick 
$j_{1} \in \Lambda_i^{*+} \setminus \Lambda_i^{-+}$ and $j_{2} \in \Lambda_i^{*-} \setminus \Lambda_i^{--}$ such that at least one of the indices lies in $\Lambda_{i}^{++} \cup \Lambda_{i}^{+-}$.

\item If $\Lambda_i^{*-}= \Lambda_i^{--}$, then $\Lambda_{i}^{--}  \neq \varnothing$  and $\Lambda_{i}^{++} \neq \varnothing$,  as $\mS$ is feasible and  \eqref{eq:omega_small} does not hold. Hence  $i \in I^{+}$ and as $i\notin J$, $\rho\notin \mD_i$ and   we can find $j_{1} \in \Lambda_i^{++}$ and $j_{2} \in \Lambda_i^{--}$ such that 
\[
\tfrac{\sigma_{2j_2}}{\sigma_{1j_2}}\, e^{\rho_{s+j_{2}}} < \tfrac{\sigma_{2j_1}}{\sigma_{1j_1}}\, e^{\rho_{s+j_{1}}}\, .\]

\item  If
    $\Lambda_i^{*+} = \Lambda_i^{-+}$,  then $\Lambda_{i}^{-+}  \neq \varnothing$ and $\Lambda_{i}^{+-} \neq \varnothing$, again as $\mS$ is feasible and  \eqref{eq:omega_small} does not hold.  Hence $i \in I^{-}$ and as $\rho\notin \mD_i$,  we can pick
$j_{1} \in \Lambda_{i}^{-+}$ and $j_{2} \in \Lambda_{i}^{+-}$ such that 
\[ 
\tfrac{\sigma_{2j_2}}{\sigma_{1j_2}}\, e^{\rho_{s+j_{2}}} < \tfrac{\sigma_{2j_1}}{\sigma_{1j_1}}\, e^{\rho_{s+j_{1}}} \, .\]
\end{enumerate}

These three cases cover all remaining cases. Note that $\sigma_{1j_1}, \sigma_{1j_2}\neq 0$ in cases (2) and (3). 
We define
   \[Q(\omega)_{ij_{1}} \coloneqq \wQ_{ij_{1}} + \omega_i\, ,  \qquad 
   Q(\omega)_{ij_{2}} \coloneqq \wQ_{ij_{2}} - \frac{\mu_{j_{1}}}{\mu_{j_{2}}}\, \omega_i, \qquad
    Q(\omega)_{ij} \coloneqq \wQ_{ij}, \quad \text{if }j\neq j_1,j_2 .\]
It is easily seen that the conditions in  \eqref{eq:row_mod} hold. We further have
\begin{equation}
\label{eqthetaMainLemma}
\sum_{j\colon P^\sigma_{ij}\mS_{ij}\neq 0\ }\frac{Q(\omega)_{ij}\mu_j}{e^{\rho_i} - \frac{\sigma_{2j}}{\sigma_{1j}}\, e^{\rho_{s + j}}}  
= 
\sum_{j\colon P^\sigma_{ij}\mS_{ij}\neq 0\ }\frac{\wQ_{ij}\mu_j}{e^{\rho_i} - \frac{\sigma_{2j}}{\sigma_{1j}}\, e^{\rho_{s + j}}}   + \theta_i(\rho,\mu)\omega_i \, 
\end{equation}
for a function $\theta_i(\rho,\mu)$ that we now describe, arising from   the indices $j_1,j_2$. 
We will use that if $\sigma_{1j}\neq 0$,  then
\begin{equation}\label{eq:signs}
\sign(e^{\rho_i} - \tfrac{\sigma_{2j}}{\sigma_{1j}}\, e^{\rho_{s + j}}) =\sign(P^{\sigma}_{ij}) \sign((A_\rho^\sigma)_{ij})=
\sign(P^{\sigma}_{ij})\sign(\mS_{ij})\, .
\end{equation}

In case (1), if one of $j_1,j_2$ does not belong to $\Lambda_{i}^{++} \cup \Lambda_{i}^{+-}$, then we have
\[ \theta_i(\rho,\mu) = 
\begin{cases}
\frac{\mu_{j_1}}{e^{\rho_i} - \frac{\sigma_{2j_1}}{\sigma_{1j_1}}e^{\rho_{s + j_1}}} & \quad \text{if } j_1 \in  \Lambda_i^{++}, \ j_2\notin  \Lambda^{+-}_i \text{ (as $P_{ij_2}^\sigma=0$)} \\[15pt]
-\frac{\mu_{j_1}}{e^{\rho_i} - \frac{\sigma_{2j_2}}{\sigma_{1j_2}}e^{\rho_{s + j_2}}} &  \quad \text{if }  j_1 \notin  \Lambda_i^{++}, \ j_2 \in  \Lambda_i^{+-} \text{ (as $P_{ij_1}^\sigma=0$)}\,.
\end{cases} \]
In both cases, $\theta_i(\rho,\mu)$ is positive by \eqref{eq:signs}. In the remaining cases, 
 \[ 
 \theta_i(\rho,\mu) = \frac{\mu_{j_{1}}}{e^{\rho_i} - \frac{\sigma_{2j_1}}{\sigma_{1j_1}}e^{\rho_{s+j_{1}}}} - \frac{\mu_{j_{1}}}{e^{\rho_i} - \frac{\sigma_{2j_2}}{\sigma_{1j_2}}e^{\rho_{s+j_{2}}}} 
 = \frac{\mu_{j_{1}}(\frac{\sigma_{2j_1}}{\sigma_{1j_1}}e^{\rho_{s+j_{1}}} - \frac{\sigma_{2j_2}}{\sigma_{1j_2}}e^{\rho_{s+j_2}})}{(e^{\rho_i} - \frac{\sigma_{2j_1}}{\sigma_{1j_1}}e^{\rho_{s+j_{1}}})(e^{\rho_i} - \frac{\sigma_{2j_2}}{\sigma_{1j_2}}e^{\rho_{s+j_{2}}})}\,. 
 \]
For the cases (2) and (3),  the denominator of $ \theta_i(\rho,\mu) $ is positive  as 
the signs of  the factors agree by \eqref{eq:signs}. The numerator is also positive by assumption. In the case (1) with $j_1,j_2\in \Lambda_{i}^{++} \cup \Lambda_{i}^{+-}$, the denominator is negative as the signs of the factors differ. The numerator is also negative, as by \eqref{eq:signs}
\[
\frac{\sigma_{2j_1}}{\sigma_{1j_1}}e^{\rho_{s+j_1}} < e^{\rho_i} < \frac{\sigma_{2j_2}}{\sigma_{1j_2}}e^{\rho_{s+j_2}}\, .
\]

 Therefore, the coefficient of $\omega_i$ in \eqref{eqthetaMainLemma} is positive, which implies that   we can choose $\omega_{i}$ large enough such that $\sum_{j\colon P^\sigma_{ij}\mS_{ij}\neq 0\ }\frac{Q(\omega)_{ij}\mu_j}{e^{\rho_i} - \frac{\sigma_{2j}}{\sigma_{1j}}e^{\rho_{s + j}}}   +\sum_{j \colon \mS_{ij}=0}P_{ij}^\sigma  \mu_j  >0$. 

 To summarise, we have shown that there exists $\mu\in \R^\ell_{>0}$ such that $P^\sigma \mu>0$ and for a suitably chosen $\omega \in \mathbb{R}^{s}_{>0}$,  $Q(\omega)\mu = 0$, 
 $\sign(Q(\omega))=\mS$,  and \eqref{eq2MainLemma} holds for all $i \in [s]$.
\end{proof}

Note that the proof of \Cref{mainlemma} does not rely on $\rho$ belonging to $\mathcal{G}^\sigma$. The sparsity of $P$ we referred to earlier is characterized in terms of the bipartite graph defined by the rows and columns of $P$.

\begin{definition}
\label[definition]{forest}
Given  $P \in \R^{s\times \ell}$,  the  bipartite graph $ G_P = (V_P, E_P) $ associated with  $P$ is defined by:
\begin{itemize}
    \item $V_P = V_1 \sqcup V_2$ with $V_1 = \{1, \dots, s\}$ and $V_2 = \{1, \dots, \ell\}$.
    \item $(i, j) \in E_P$ for $i\in V_1$ and $j\in V_2$ if and only if $P_{i j} \neq 0$. 
\end{itemize}
We say that $P$ \term{induces a forest} if  $G_P$ is a forest, that is, each connected component  is a tree (a connected acyclic graph).
\end{definition}

Although the condition that $P$ induces a forest might  in general be quite restrictive, it holds for the (simplified) reduced matrices of many augmented vertical systems coming from (realistic) reaction networks. As the graph $G_P$ associated with a simplified reduced matrix $P$ can be seen as a subgraph of the graph $G_{\oP}$ 
associated with the reduced matrix $\oP$ from \eqref{eqP}, if $\oP$ induces a forest, then so does $P$. 
 
\begin{theorem}
\label[theorem]{thm:multiSimplified}
 Let $F  \in \R[a,b,x^\pm]^d$ be an augmented vertical system defined by a principal matrix. Consider a choice of row and column partitions and let  $P\in \R^{s\times \ell}$ be  the  associated simplified reduced matrix  from \eqref{def:simpif_reduced_matrix}. 
  Assume that $P$ induces a forest. 
  
For every orientation $\sigma \in \{-1,0,1\}^{2 \times \ell}$   compatible with the column partition  it holds   that 
\[ \mE^\sigma_F\  = \ \bigcup_{\mS} \mathcal{C}_{\mathcal{S}}^\sigma\, , \]
where the union is over all feasible matrices $\mathcal{S}\in \{-1,0,1\}^{s\times \ell}$.
  \end{theorem}
\begin{proof} 
The inclusion $\subseteq$ was shown in \Cref{thm:mono}. For the reverse inclusion, given $\rho \in \mathcal{C}_{\mathcal{S}}^\sigma$ with $\mS$ feasible,  $\rho \in \mE^\sigma_F$ follows if we show that the oriented characteristic system associated with $(P,\sigma,\rho)$ has a solution.

\cref{mainlemma} tells us that there exist $\bar{\mu}\in \R^\ell_{>0}$  and $Q \in \mathbb{R}^{s \times \ell}$ with $\sign(Q) = \mathcal{S}$ such that   $Q \bar{\mu} = 0$, $P^\sigma\bar{\mu} > 0$, and 
\[
      \sum_{j\colon P^\sigma_{ij}\mS_{ij}\neq 0\ }\frac{Q_{ij} \, \bar{\mu}_j}{e^{\rho_i} - \frac{\sigma_{2j}}{\sigma_{1j}}\,e^{\rho_{s + j}}}  + \sum_{j\colon \mS_{ij}=0}P^{\sigma}_{ij}  \bar{\mu}_j  > 0  \quad \text{for all }i \in [s]\,.
\]
     We will now modify $\bar{\mu}$ to obtain  $\mu\in \R^\ell_{>0}$ satisfying $A_\rho^\sigma \mu=0$ and $P^\sigma\mu>0$.

 First of all, observe that 
as  $(A_\rho^\sigma)_{ij}=0$ and $P^\sigma_{ij}=0$ if $i$ and $j$ are not in the same connected component of $G_P$, 
  the matrices $A_\rho^\sigma$ and $P^\sigma$ have a block structure given by the set of nodes in each connected component of $G_P$. This means that, without loss of generality, we can assume that $G_P$ is connected and hence a tree.

Consider the node $1\in [s]$ of $G_P$ and let $\omega_1=1$. 
 As the graph is acyclic, for every node $i\in  [s]$, $i\neq 1$, there is a unique path joining $1$ and $i$
 \[ 1= i_0  \mathdash j_0  \mathdash i_1  \mathdash j_1  \mathdash \cdots  \mathdash i_{u}  \mathdash j_u  \mathdash i_{u+1}= i  \]
 with $i_k\in [s]$, $j_k\in [\ell]$ and   without repeated nodes. Let 
 \[ \omega_i \coloneqq \prod_{k=0}^u \ \frac{Q_{i_k,j_k}}{(A_\rho^\sigma)_{i_k,j_k}}\,  \frac{(A_\rho^\sigma)_{i_{k+1},j_k}}{Q_{i_{k+1},j_k}} \   > 0\, , \]
 with the usual convention that $\tfrac{0}{0}=1$. As $\sign(A_\rho^\sigma)=\sign(Q)$, we are guaranteed that $\omega_i$ is positive. 
 
 For $j\in [\ell]$, let $\iota(j)\in [s]$ be the only node that is adjacent to $j$ in the unique path from $1\in [s]$ to $j$. 
We define 
  \begin{equation*}\label{equationmu}
  \mu_j \coloneqq   \omega_{\iota(j)} \,  \frac{Q_{\iota(j), j}}{(A_{\rho}^\sigma)_{\iota(j),j}}\, \bar{\mu}_j  \quad > 0 \, .
      \end{equation*}

Let us see that $A_\rho^\sigma\mu=0$ and $P^\sigma \mu>0$. 
Fix $i\in [s]$. For a given $j\in [\ell]$ with $P_{ij}\neq 0$, if $\iota(j)\neq i$, then  the unique path from $1$ to $i$ is precisely obtained by appending $\iota(j) \mathdash j \mathdash i$ to the unique path from $1$ to $\iota(j)$. Hence
\[ \omega_{i} = \omega_{\iota(j)}\,  \frac{Q_{\iota(j),j}}{(A^\sigma_\rho)_{\iota(j),j}}\,  \frac{(A^\sigma_\rho)_{ij}}{Q_{ij}} \quad \Rightarrow \quad 
 \omega_{\iota(j)} = \omega_{i} \,  \frac{(A_\rho^\sigma)_{\iota(j),j}}{Q_{\iota(j),j}} \,  \frac{Q_{ij}}{(A^\sigma_\rho)_{ij}}\, .  \] 
Using this, we have that for any matrix $B\in \R^{s\times \ell}$ with support included in the support of $P$, it holds
\begin{align*}
(B \mu)_i & = \sum_{ j=1}^\ell  B_{ij}  \, \omega_{\iota(j)} \,  \frac{Q_{\iota(j), j}}{(A^\sigma_{\rho})_{\iota(j),j}}\, \bar{\mu}_j \\ & =  
\sum_{j=1 , \iota(j)=i}^\ell  B_{ij}  \, \omega_{i} \,  \frac{Q_{i j}}{(A^\sigma_{\rho})_{ij}}\,   \bar{\mu}_j  +
\sum_{j=1, \iota(j)\neq i }^\ell B_{ij}  \,\omega_{i}  \,  \frac{(A^\sigma_\rho)_{\iota(j),j}}{Q_{\iota(j),j}} \,  \frac{Q_{ij}}{(A^\sigma_\rho)_{ij}}  \,  \frac{Q_{\iota(j), j}}{(A^\sigma_{\rho})_{\iota(j),j}}\, \bar{\mu}_j
\\ & =  
\omega_{i}  \sum_{ j=1}^\ell  B_{ij}  \,   \frac{Q_{ij}}{(A^\sigma_{\rho})_{ij}}\,   \bar{\mu}_j \, .
\end{align*}
Considering $B=A_\rho^\sigma$, we obtain 
\begin{align*}
(A^\sigma_\rho \mu)_i & = \omega_{i}  \sum_{ j=1}^\ell  (A^\sigma_\rho)_{ij}  \,   \frac{Q_{ij}}{(A^\sigma_{\rho})_{ij}}\,   \bar{\mu}_j
 = 
\omega_{i} \sum_{j=1}^\ell \,  \,  Q_{ij} \bar{\mu}_j  =0\, .
\end{align*}
Similarly, for $B=P^\sigma$, we have
 \begin{align*}
(P^\sigma \mu)_i & = \omega_{i}  \sum_{ j=1}^\ell  P^\sigma_{ij}  \,   \frac{Q_{ij}}{(A^\sigma_{\rho})_{ij}}\,   \bar{\mu}_j 
=
 \omega_{i}\  \Big( \ \sum_{ j\colon P^\sigma_{ij}\mS_{ij}\neq 0\ }    \,    \frac{Q_{ij}}{ e^{\rho_i} -\tfrac{\sigma_{2j}}{\sigma_{1j}}\,  e^{\rho_{s+j}} }\,   \bar{\mu}_j 
 + \sum_{ j\colon \mS_{ij}=0}    \,   P^\sigma_{ij}  \,   \bar{\mu}_j\ \Big)  > 0 \, .
 \end{align*} 
This concludes the proof. 
\end{proof}

\begin{remark}
\cref{thm:multiSimplified} generalizes Proposition 2.11.6 of \cite{Ji2011UniquenessOE}, when translated into our terminology.
\end{remark}

Combining our results, we obtain the following characterization of augmented vertical systems that admit multiple positive zeros.

\begin{theorem}[Characterization of multiple positive zeros when $P$ induces a forest]\label[theorem]{thm:full}
 Let $F  \in \R[a,b,x^\pm]^d$ be an augmented vertical system defined by a principal matrix. Consider a choice of row and column partitions and let  $P\in \R^{s\times \ell}$ be  the  associated simplified reduced matrix  from \eqref{def:simpif_reduced_matrix}. 
  Assume that $P$ induces a forest. 
 
 Then, $F$ admits multiple positive zeros  if and only if $ \mathcal{C}_{\mS}^\sigma \neq \varnothing$ for an orientation $\sigma \in \{-1,0,1\}^{2 \times \ell}$ compatible with the column partition and a feasible matrix  $\mathcal{S}\in \{-1,0,1\}^{s \times \ell}$.
\end{theorem}
\begin{proof}
This follows from 
\Cref{thm:mainsimplification,thm:first_characterization,thm:multiSimplified}.
\end{proof}

When considering the singleton partitions,  \cref{rk:singleton} and the definition of $\mathcal{C}_{\mathcal{S}}^{\sigma_+}$ from \Cref{def:feasible_ground_set} give the following corollary that applies when $\oP$ induces a forest. 

\begin{corollary}[Characterization of multiple positive zeros when $\oP$ induces a forest]\label[corollary]{cor:full}
 Let $F=( C (a\star x^M), Lx-b) \in \R[a,b,x^\pm]^d$ be an augmented vertical system with $C\in \R^{\bs\times \bar{m}}$ principal and $L$ of full row rank.  Assume that the reduced matrix  $\oP\in \R^{\bs\times \bl}$ from \eqref{eqP} induces a forest.
  The following statements are equivalent:
 \begin{enumerate}[label=(\roman*)]
 \item $F$ admits multiple positive zeros.
 \item There exists a feasible matrix $\mathcal{S}\in \{-1,0,1\}^{\bs \times \bl}$  such that $\mathcal{C}_{\mathcal{S}}^{\sigma_+}\neq \varnothing$. In other words, there exists $\rho\in M^\top(\mathcal{O}_L)$ such that $\sign(A_\rho)=\mS$ and 
$\Gamma_{J} \neq \varnothing$ for 
 $J=\{i \in I^+\cup I^-: \rho\in \mathcal{D}_i\}$.
\end{enumerate}
\end{corollary}

\begin{remark}
If $P$ induces a forest, then $\mC_\mS^\sigma\neq \varnothing$
implies   $\mE_F\neq \varnothing$.  Hence by \Cref{thm:mono}, the set $\mC_{\mS'}^{\sigma_+}$ obtained by considering the row and column singleton partitions, is also nonempty for some feasible sign matrix $\mS'\in \{-1,0,1\}^{\bar{s}\times \bar{\ell}}$. 
The reverse implication might not hold if $\oP$ does not induce a forest.     
\end{remark}

\begin{remark}[Witnesses of multiple positive zeros]\label[remark]{rk:witness}
In the setting of \Cref{thm:full}, given suitable $\mS$, $\sigma$, and $\rho\in \mathcal{C}_{\mS}^\sigma$,  we can find parameter values $a,b$ for which the system $F_{a,b}$ has two distinct positive zeros. 
Explicitly, we follow these steps:
\begin{itemize}
    \item As $\rho\in \mE^\sigma_F\subseteq  \mG^\sigma$, \Cref{thm:mainsimplification} gives  a $\bar{\rho}\in \mE_F\subseteq M^\top(\mO_L)$. We will see in \Cref{sec:comput} that $\delta\in \mathcal{O}_L$ such that $\bar{\rho}=M^\top(\delta)$ is readily obtained when finding $\rho$. 
    \item Find $v\in \ker(L)$ such that $\sign(v)=\sign(\delta)$. 
     \item Find a solution $\mu\in \R^{\bar{\ell}}_{>0}$ to  the  characteristic system at $\bar{\rho}$.
   
\item A witness for multiple positive zeros is given by $\Phi(v,\delta,\mu)$ with $\Phi$ from \eqref{eq:Phi}. 
\end{itemize}  

This procedure is illustrated in \Cref{example:hybrid_histidine_kinase}. 
\end{remark}

\section{Examples}
\label{sec:examples}
In this section we illustrate, through detailed examples, how to apply \Cref{thm:full,thm:mono,cor:full}
to determine whether an augmented vertical system admits multiple  positive zeros. 

\begin{example}
    Let $F = C(a \star x^M)$ be the vertical system with
    \[ C = \begin{pmatrix}
        1 & 0 & -1 & 1 & 0 \\
        0 & 1 & -1 & 1 & -1 
    \end{pmatrix}\,, \quad M = \begin{pmatrix}
        0 & -1 & 1 & 2 & 2 \\
        0 & 0 & 1 & 1 & 1
    \end{pmatrix}\,.\]
We  have in particular that $n=2$, $\bs=2$,  $\bar{m}=5$, $L$ is empty hence $\mO_L=\R^2\setminus \{0\}$, and the matrix $\oP$  from \eqref{eqP} is
\[  \oP = \begin{pmatrix}
         1 & -1 & 0 \\
        1 & -1 & 1 
    \end{pmatrix}\, 
.\]
We consider the simplified reduced matrix
\[P = \begin{pmatrix} 
 1 & 0 \\
 1 & 1 \end{pmatrix}
 \]
obtained from the column partition $\alpha$ with blocks $\{1,2\},\{3\}$ and 
the row partition being the singleton partition; hence $s = 2$, $\ell=2$, $m=4$. 
We now  use \Cref{thm:mono} to show  that $F$ does not admit multiple positive zeros. That is, we show that $\mC_{\mS}^\sigma=\varnothing$ for all orientations $\sigma\in \{-1,0,1\}^{2\times 2}$ compatible with $\alpha$ and all feasible sign matrices $\mS$. 

As $\mathcal{U}_\alpha=\{2\}$, 
an orientation is compatible with $\alpha$ if and only if 
$\sigma_{1,2}=\sigma_{2,2}=1$. 
The first row of $P^\sigma$ has one nonzero entry equal to $\sigma_{1,1}$, and hence $P^\sigma \mu>0$ cannot hold for $\mu>0$ unless $\sigma_{1,1}> 0$. Hence, if $\sigma_{1,1}\leq 0$, $\mC_{\mS}^\sigma=\varnothing$ for any sign matrix $\mS$ as $\Gamma_J=\varnothing$ for all  sets $J$. We are left to consider the following orientations compatible with $\alpha$: 
\[ \begin{pmatrix}
    1 & 1 \\  1 & 1 
\end{pmatrix}\,,  \quad   \begin{pmatrix}
    1 & 1 \\  0 & 1 
\end{pmatrix}\, , \quad \begin{pmatrix}
    1 & 1 \\  -1 & 1 
\end{pmatrix}\, .     \]
For these three orientations   we have
\[ A_\rho^\sigma = 
\begin{pmatrix}  
\sigma_{1,1} e^{\rho_1} - \sigma_{2,1} e^{\rho_3} &  0 \\ 
\sigma_{1,1} e^{\rho_2} - \sigma_{2,1}e^{\rho_3}  &  e^{\rho_2} - e^{\rho_4}
\end{pmatrix}  \, . \]
For the second and third orientations, $\sign(A_\rho^\sigma)$ is not feasible, as the first row has exactly one nonzero entry. Hence, for any such  orientation, $\mC_{\mS}^\sigma=\varnothing$ for all feasible sign matrices $\mS$.

 All that is left is to consider the  first orientation, namely $\sigma=\sigma_+$, giving
 \[ A_{\rho}^{\sigma} = \begin{pmatrix}
    e^{\rho_1} - e^{\rho_3} & 0 \\
    e^{\rho_2} - e^{\rho_3} & e^{\rho_2} - e^{\rho_4}
 \end{pmatrix}\,. \]
Using that $M^\top(\delta)=(0,-\delta_1,\delta_1+\delta_2,2\delta_1+\delta_2,2\delta_1+\delta_2)$ and \Cref{def:mGsigma},  
it holds that $\rho \in \mG^{\sigma}$ if and only if there exists a nonzero $\delta \in \R^2$ and $z_+,z_-$ such that
 \begin{equation}
 \begin{aligned}
     \label{eq:above}
         &\rho_1 = 0   \qquad \rho_2 = -\delta_1 \qquad z_+ = \delta_1 + \delta_2 \qquad z_- = 2\delta_1 + \delta_2 \qquad \rho_4 = 2\delta_1 + \delta_2 \\
         &\sign(z_+-z_-)=\sign(\rho_3-z_+)=\sign(\rho_3-z_-)\,.
     \end{aligned}
     \end{equation}

If $\mS$ is feasible and $\sign(A_{\rho}^{\sigma})=\mS$, then necessarily the first row of $\mS$ is identically zero. This leaves us with the following three feasible sign matrices to consider:
\[ \mS_1 = \begin{pmatrix}
     0 & 0 \\
     0 & 0
 \end{pmatrix} \qquad \mS_2 = \begin{pmatrix}
     0 & 0 \\
     1 & - 1
 \end{pmatrix} \qquad \mS_3 = \begin{pmatrix}
     0 & 0 \\
     -1 & 1
 \end{pmatrix}\, . \]
For each of these matrices, the condition $\sign(A_{\rho}^{\sigma}) = \mS$ is incompatible with \eqref{eq:above}:
\begin{itemize}
    \item For $\mS=\mS_1$, the condition gives $\rho_1=\rho_2=\rho_3=\rho_4$, which cannot hold  together with \eqref{eq:above} for a nonzero $\delta$. 
    \item For $\mS=\mS_2$, the condition gives $\rho_1 = \rho_3 < \rho_2 < \rho_4$, 
which combined with \eqref{eq:above} implies
$0=\rho_3$, $0<\rho_2 = z_+-z_-$, and $0<\rho_4=z_-$. This contradicts that the signs of $z_+-z_-$ and of $\rho_3-z_-=-z_-$ agree. 
\item For $\mS=\mS_3$ it holds similarly  that
 the constraint $\rho_1 = \rho_3 > \rho_2 > \rho_4$ and \eqref{eq:above} cannot both hold. 
\end{itemize}
 Hence $\mC_{\mS}^\sigma=\varnothing$ for $\mS=\mS_1,\mS_2,\mS_3$.
We have thus shown that $\mC_{\mS}^\sigma=\varnothing$ for all orientations $\sigma$ compatible with $\alpha$ and feasible sign matrices $\mS$, and we conclude from  \Cref{thm:mono}   that $F$ does not admit multiple positive zeros.
\end{example}

\begin{example}\label[example]{example:hybrid_histidine_kinase}
We illustrate \Cref{thm:full} with an augmented vertical system that arises from the study of the steady states of a reaction network named 
\emph{hybrid histidine kinase}. This network is known to admit multiple positive steady states \cite{hhk, plos:identifying}. The stoichiometric matrix, the matrix of reactants, and a matrix defining the stoichiometric compatibility classes are
{\small \[ N = \begin{pmatrix}
    -1 & 0 & 0 & 1 & 0 & 0 \\
    1 & -1 & 0 & 0 & 1 & 0 \\
    0 & 1 & -1 & -1 & 0 & 0 \\
   0 & 0 & 1 & 0 & -1 & 0 \\
   0 & 0 & 0 & -1 & -1 & 1 \\
   0 & 0 & 0 & 1 & 1 & -1 \\
\end{pmatrix}, \quad M = \begin{pmatrix} 
1 & 0 & 0 & 0 & 0 & 0 \\
0 & 1 & 0 & 0 & 0 & 0 \\
0 & 0 & 1 & 1 & 0 & 0 \\
0 & 0 & 0 & 0 & 1 & 0 \\
0 & 0 & 0 & 1 & 1 & 0 \\
0 & 0 & 0 & 0 & 0 & 1\end{pmatrix},   \quad L = \begin{pmatrix} 
1 & 1 & 1 & 1 & 0 & 0 \\
0 & 0 & 0 & 0 & 1 & 1\end{pmatrix}.\] }%
After performing Gaussian elimination on $N$, the steady states of the network are the zeros of the augmented vertical system 
$F = (C(a \star x^M), Lx - b) \in \R[a,b,x]^6$
with $C$ and the matrix $A_{\bar{\rho}}$ from \eqref{eq:characteristic_matrix} for $\bar{\rho}\in \R^6$ given by
 \[ C  = \begin{pmatrix} 
1 & 0 & 0 & 0 & 1 & -1 \\
0 & 1 & 0 & 0 & 0 & -1 \\
0 & 0 & 1 & 0 & -1 & 0 \\
0 & 0 & 0 & 1 & 1 & -1 \end{pmatrix}\,,\qquad A_{\bar{\rho}} = \begin{pmatrix} 
-e^{\bar{\rho}_1} + e^{\bar{\rho}_5} & e^{\bar{\rho}_1} - e^{\bar{\rho}_6} \\
0 & e^{\bar{\rho}_2} - e^{\bar{\rho}_6} \\
e^{\bar{\rho}_3} - e^{\bar{\rho}_5} & 0 \\
-e^{\bar{\rho}_4} + e^{\bar{\rho}_5} & e^{\bar{\rho}_4} - e^{\bar{\rho}_6} 
\end{pmatrix}  .\]
The last two columns of $C$ form the matrix $-\oP \in \R^{4 \times 2}$ from \eqref{eqP}. In this example we have $\bs=4$,  $\bar{m}=6$. We consider the row and column partitions to be $[\bs] = \{1,4\} \sqcup \{2\} \sqcup \{3\}$ and $[\bl] = \{1\} \sqcup \{2\}$, so $s = 3$,  $\ell=2$, $m=5$. As the column partition consists of singletons, $\sigma_+$ is the only compatible orientation. We obtain the following simplified reduced matrix $P$,  matrix $A^{\sigma_+}_\rho$ for $\rho\in \R^5$, and associated graph $G_{P}$:

\begin{minipage}[h]{0.7\textwidth}
\[P = \begin{pmatrix}
-1 & 1 \\
0 & 1 \\
1 & 0 \\ \end{pmatrix}\,, \qquad A^{\sigma_+}_{\rho} = \begin{pmatrix} 
-e^{\rho_1} + e^{\rho_4} & e^{\rho_1} - e^{\rho_5} \\
0 & e^{\rho_2} - e^{\rho_5} \\
e^{\rho_3} - e^{\rho_4} & 0\end{pmatrix}\, ,   \]
\end{minipage}\quad
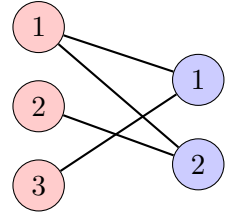
\begin{minipage}[h]{0.25\textwidth}
\begin{tikzpicture}[baseline=(current bounding box.center), scale=0.7]
    \foreach \i/\y in {1/0, 2/-1.5, 3/-3} {
        \node[draw, circle, fill=red!20, minimum size=5mm] (left\i) at (0, \y) {\i};
    }
    
    \foreach \j/\y in {1/-1, 2/-2.6} {
        \node[draw, circle, fill=blue!20, minimum size=5mm] (right\j) at (3, \y) {\j};
    }
    
    \draw[thick] (left1) -- (right1);
    \draw[thick] (left1) -- (right2);
    \draw[thick] (left2) -- (right2);
    \draw[thick] (left3) -- (right1);
\end{tikzpicture}
\end{minipage} 

\smallskip
As $G_{P}$ is acyclic, $P$ induces a forest and 
\Cref{thm:full} completely characterizes whether $F$ admits multiple positive zeros. 

We now study the sets $\mathcal{C}_{\mathcal{S}}^{\sigma_+}$ from  \Cref{def:feasible_ground_set} for different sign matrices $\mS$ to determine whether at least one of them is nonempty. Observe that $\sign(A_\rho^{\sigma_+})$ is feasible only when the bottom two rows of $A_\rho^{\sigma_+}$ are zero. 
The condition
$\rho \in \mathcal{G}^{\sigma_+}$, see \Cref{rk:singleton}, requires that, for some $\delta \in \mathcal{O}_{L}$, it holds 
    \begin{equation}\label{eq:rho_delta_example}
        \rho_1 = \delta_1 = \delta_3 + \delta_5\,, \quad \rho_2 = \delta_2\,, \quad \rho_3 = \delta_3\,,\quad    \rho_4 = \delta_4 + \delta_5\,,\quad\rho_5 = \delta_6\,.
    \end{equation}

We start by considering the feasible sign matrix 
\[\mS = \begin{pmatrix}
    0 & 0 \\
    0 & 0 \\
    0 & 0
\end{pmatrix}\, .\]
The condition $\sign(A_\rho^{\sigma_+}) =\mS$ gives that all entries of $\rho$ are equal and in particular $\rho_1= \rho_3$. Using \eqref{eq:rho_delta_example}, we obtain that $\delta_5=0$. Since $\delta\in \mO_L$, also $\delta_6=0$. This gives that $\rho=0$, hence $\delta=0$, a contradiction. 
Hence $\mC_{\mS}^{\sigma_+} = \varnothing$.

We consider now the feasible sign matrix
\[\mathcal{S} = \begin{pmatrix} 
-1 & 1 \\
0 & 0  \\
0 & 0 \end{pmatrix}\, . \] 
We easily find from \Cref{def:Subsets,def:ID} that $\Lambda_1^{++}=\{2\}$, $\Lambda_1^{--}=\{1\}$, $I^+=\{1\}$, and $I^-=\varnothing$.
As $\mS_{1j}\neq 0$ for $j=1,2$, it holds $\Gamma_{I^+} = \varnothing$  and hence the only subset  to consider for $\mathcal{C}_{\mathcal{S}}^{\sigma_+}$ is $J=\varnothing$. 
This gives that $\rho\in \mathcal{C}_{\mathcal{S}}^{\sigma_+}$ if and only if  \eqref{eq:rho_delta_example} and 
the following statements hold:
\begin{itemize}
    \item $\sign(A_\rho^{\sigma_+}) =\mS$, that is, 
    \[ \rho_1 >\rho_4\,,\quad\rho_1> \rho_5 \,,\quad \rho_2 = \rho_5\,,\quad \rho_3 = \rho_4\,.\] 
      \item $\varnothing=\{i \in I^+\cup I^-: \rho\in \mathcal{D}_i\}$, that is, $\rho\notin \mathcal{D}_1$: 
    \[\rho_5 = \max_{j \in \Lambda_1^{++}}\rho_{3 + j} >  \min_{j \in \Lambda_1^{--}}\rho_{3 + j} = \rho_4 \, . \]
    \end{itemize}

These conditions allow us to express  $\mathcal{C}_{\mathcal{S}}^{\sigma_+}$ as
\begin{multline*}
\mathcal{C}_{\mathcal{S}}^{\sigma_+} = \{\rho \in \mathbb{R}^5 \colon  \text{there exists } \delta \in \mathcal{O}_{L} \text{ such that } \rho_1 > \rho_5 > \rho_4, \rho_2 = \rho_5, \rho_3 = \rho_4, \\ \rho_1 = \delta_1 = \delta_3 + \delta_5, \: \rho_2 = \delta_2, \: \rho_3 = \delta_3,   \rho_4 = \delta_4 + \delta_5, \: \rho_5 = \delta_6 \}\,.
\end{multline*}
and hence $\mathcal{C}_{\mathcal{S}}^{\sigma_+}\neq \varnothing$ if and only if 
there exists $\delta \in \mathcal{O}_{L}$ such that
\[  \delta_3 + \delta_5 = \delta_1  >  \delta_6 = \delta_2 > \delta_4 + \delta_5= \delta_3 \, .  \]
These relations hold for 
\[ \delta \coloneqq (\ln 2,\ln \tfrac{1}{2},\ln \tfrac{2}{5},\ln \tfrac{2}{25},\ln 5,\ln \tfrac{1}{2})\, ,  \] 
and furthermore $\delta\in \mO_L$ as
 $\sign(\delta)=\sign(v)$
with $v= (3,-1,-1 ,-1,1,-1) \in \ker(L)$. 
Hence, $\mC_{\mS}^{\sigma_+}\neq \varnothing$ and by \Cref{thm:full}, $F$ admits multiple positive zeros.

 We now follow \Cref{rk:witness} to find parameters $(a,b)$ and a pair of distinct positive zeros $x,y$ of $F_{a,b}$ for the chosen 
 $\rho\in \mathcal{C}_{\mathcal{S}}^{\sigma_+}$
 and the chosen $\delta$.  We find that 
 \[\bar{\rho}=M^\top(\delta)=(\ln 2,\ln \tfrac{1}{2},\ln \tfrac{2}{5},\ln 2,\ln \tfrac{2}{5},\ln \tfrac{1}{2})\] and  that $\mu= (\tfrac{3}{2},\tfrac{8}{5})$ is a solution to the  characteristic system
\begin{align*}
 0 &=  - \tfrac{8}{5} \mu_1 + \tfrac{3}{2}\mu_2=  (-e^{\bar{\rho}_1} + e^{\bar{\rho}_5})\mu_1 + (e^{\bar{\rho}_1} - e^{\bar{\rho}_6})\mu_2  = (-e^{\bar{\rho}_4} + e^{\bar{\rho}_5})\mu_1 + (e^{\bar{\rho}_4} - e^{\bar{\rho}_6})\mu_2  \, , \\
 0 & < -\mu_1 + \mu_2 \, , \qquad 0<\mu_1\, , \qquad 0<\mu_2\, .
\end{align*}
Hence  $(a,x,y) = \Phi(v,\delta,\mu) \in \mM_F$: 
\[ 
a = \big(\tfrac{1}{30},\tfrac{4}{5}, \tfrac{9}{10},\tfrac{6}{25},\tfrac{138}{25}, \tfrac{4}{5}\big)\, , \qquad x = \big(6,1,\tfrac{2}{3},\tfrac{2}{23},\tfrac{5}{4},1\big)\, , \qquad y = \big(3,2,\tfrac{5}{3},\tfrac{25}{23},\tfrac{1}{4},2\big)\, . \]
 By letting $b = Lx=(\tfrac{535}{69},\tfrac{9}{4}) \in \mathbb{R}^2$, we obtain that $x,y$ are distinct positive zeros of $F_{a,b}$. Other choices of $v \in \ker(L)$ with $\sign(v)=\sign(\delta)$ may lead to different parameters and pairs of distinct zeros. 
\end{example}

 \begin{example}
Consider  the  vertical system 
\[F =(a_1x_1x_2-a_3x_2 - a_4x_1^2 + a_5x_2, \ a_2x_1x_2 - a_4x_1^2-a_6x_2^2)\] in $\R[a,x]^2$ defined by the matrices
\[C  = \begin{pmatrix} 
1 & 0 & -1 & -1 & 1 & 0 \\
0 & 1 & 0 & -1 & 0 & -1 
\end{pmatrix}\, \qquad
M = \begin{pmatrix}
1 & 1 & 0 & 2 & 0 & 0   \\
1 & 1 & 1 & 0 & 1 & 2 
\end{pmatrix} \, \]
and where $n=2$, $\bs=2$,  and $\bar{m}=6$. The reduced matrix $\oP$ from \eqref{eqP}
is 
\[ \oP=\begin{pmatrix} 
 1 & 1 & -1 & 0 \\
 0 & 1 & 0 & 1 
\end{pmatrix}\,, \]
which induces a forest. We now apply \Cref{cor:full}(ii) to show that $F$ admits multiple positive zeros. To this end, we need to find a feasible sign matrix $\mS$ and a nonzero $\delta\in \R^2$ such that $\rho= M^\top(\delta)$ satisfies $\sign(A_\rho)=\mS$ and 
$\Gamma_{J} \neq \varnothing$ for 
 $J=\{i \in I^+\cup I^-: \rho\in \mathcal{D}_i\}$. 

We have $\rho= M^\top(\delta)=(\delta_1+\delta_2,\delta_1+\delta_2, \delta_2,2\delta_1,\delta_2,2\delta_2)$ and  
\[ A_{\rho} = \begin{pmatrix} 
 e^{\rho_1} - e^{\rho_3} &  e^{\rho_1} - e^{\rho_4}  &  -e^{\rho_1} + e^{\rho_5} &  0 \\
 0 & e^{\rho_2} - e^{\rho_4}  & 0 & e^{\rho_2} - e^{\rho_6}  
 \end{pmatrix}\, . 
    \]
For the feasible sign matrix 
\[ \mS = \begin{pmatrix}
1 & 0 & -1 & 0 \\ 0 & 0 & 0 & 0     
\end{pmatrix}\]
we have the relations
\[ \rho_1=\rho_2=\rho_4=\rho_6\,, \quad \rho_1>\rho_3\,,\quad \rho_1>\rho_5 \,, \]
which are satisfied with $\delta= (\tfrac{1}{2},\tfrac{1}{2})$. 
For this $\mS$, we have $I^+=\{1\}$ and $I^-=\varnothing$, and the set $\mD_1$ is defined by the inequality $\rho_3\leq \rho_5$, which holds for all $\rho\in M^\top(\R^2)$.   Taking $J=\{1\}$, it is easily verified that $\Gamma_J\neq \varnothing$. 

We conclude that $\mS$ and $\delta= (\tfrac{1}{2},\tfrac{1}{2})$  satisfy the conditions of \Cref{cor:full}(ii), and hence $F$ admits multiple positive zeros. 
\end{example}

\section{Computational considerations}\label{sec:comput}

As illustrated by the examples in the previous section, to 
decide whether the feasible ground set $\mathcal{C}^{\sigma}_{\mS}$ is nonempty  we can study whether a finite collection of sets, defined in terms of linear equalities and inequalities,  is nonempty. Fix row and column partitions $\tau, \alpha$ respectively, and consider a feasible sign matrix $\mS$ and an orientation $\sigma$ compatible with the column partition. 
By \Cref{def:ID}, the sets $I^+,I^-$ are easily computed from $P,\sigma,\mS$.
 \Cref{def:feasible_ground_set} tells us that $ \mathcal{C}_{\mS}^\sigma \neq \varnothing$ if and only if there exists  $J \subseteq I^+\cup I^-$ such that 
\[ \Gamma_J\neq \varnothing \quad \text{ and }\quad  \mathcal{R}_J\coloneqq \Big\{\rho \in \mathcal{G}^{\sigma} : \sign(A_{\rho}^\sigma) = \mathcal{S},\  J=\{i \in I^+\cup I^-: \rho\in \mathcal{D}_i\} \Big\} \neq \varnothing \, .  \]

We start by noting that for every pair of real numbers $\nu,\nu' \in \R$, the inequalities of the form $-e^{\nu} < e^{\nu'}$   always hold, while those of the form $e^{\nu} < -e^{\nu'}$ are never true.  The remaining case 
$e^{\nu} < e^{\nu'}$ is equivalent to $\nu < \nu'$. 
In this way, all inequalities comparing two exponentials can be expressed as linear inequalities. 

To decide whether $\mC_\mS^\sigma\neq \varnothing$ we iteratively consider orthants $\mO\subseteq \mO_L$ and apply the following construction. 
Let $t\coloneqq \ell - \# \mathcal{U}_\alpha$ and consider the set $\mathcal{B}_{\mO}$ of all $(\rho,\delta,z_+,z_-,\mu) \in \R^m\times \mO \times \R^{t} \times \R^{t} \times \R^\ell_{>0}$ satisfying the following constraints:

\begin{enumerate}[label=(\roman*)]
    \item  Relations to impose $\rho \in \mG^\sigma$ with $\bar{\rho}=M^\top(\delta)$: 
    \begin{itemize}
    \item $\delta \in \mO$ (described by a set of linear equalities and inequalities);
    \item for all $k \in [s]$,
\begin{equation*}\label{eq:rho_delta} \rho_k = M^{\top}(\delta)_i  \quad\text{for all}  \ i\in \tau_k\,,
\end{equation*}
which give linear equalities;
\item for  all $k \in \mathcal{U}_{\alpha}$, \eqref{eq:conditionU1} holds if either 
\[ M^{\top}(\delta)_{\bar{s} + j} = \rho_{s+k}\, \quad \text{ for all }j \in \alpha_k \, , \]
or there exist $j_1,j_2 \in \alpha_k$ such that
\[ M^{\top}(\delta)_{\bar{s} + j_1} < \rho_{s+k} < M^{\top}(\delta)_{\bar{s} + j_2}\, . \] 
\item for all $k \in [\ell] \setminus  \mathcal{U}_{\alpha}$, we derive similar linear constraints for $z_+$ and $z_-$ from \eqref{eq:conditionU2} and \eqref{eq:Vsigns}.
\end{itemize}
\item $\sign(A_{\rho}^{\sigma})= \mS$ gives rise to  equalities  after imposing
\[ \sign(P_{ij}(\sigma_{1j}e^{\rho_i} - \sigma_{2j}e^{\rho_{s + j}})) = \mS_{ij} \qquad i\in [s], \ j\in [\ell]\,,  \] 
which in turn are expressed as inequalities or equalities in the entries of $\rho$. 
\item For each $J \subseteq I^+\cup I^-$, the condition $\Gamma_J\neq \varnothing$ gives rise to linear inequalities in $\mu$. To   require $J=\{i \in I^+\cup I^-: \rho\in \mathcal{D}_i\}$, we impose $i \in J$ if and only if $\rho \in \mathcal{D}_i$. So, for $i\in J$ we require
\[ \frac{\sigma_{2j_1}}{\sigma_{1j_1}}\, e^{\rho_{s + j_1}} \leq \frac{\sigma_{2j_2}}{\sigma_{1j_2}}\, e^{\rho_{s + j_2}}\]
for all pairs of indices  $(j_1,j_2)\in \Lambda_i^{++}\times \Lambda_i^{--}$ if $i\in I^+$, and analogously for $I^-$.  For $i\notin J$, we ask for the existence of such a pair for which the inequality does not hold. 
\end{enumerate}

Each $\mathcal{B}_\mO$ is a union of sets defined by linear equalities and inequalities, and hence is a union of relative interiors of    polyhedral cones. 
The set $\mC_\mS^\sigma$ is obtained by considering the linear projection 
of each $\mathcal{B}_\mO$ onto the first component $\R^m$ for all  orthants $\mO\subseteq \mO_L$. 
Hence,  $\mC_\mS^\sigma\neq \varnothing$ if and only if $\mathcal{B}_\mO\neq \varnothing$ for some orthant $\mO\subseteq \mO_L$, and this  can be decided using linear programming.
This approach gives the relevant $\delta$, useful for the construction of witnesses of multiple positive zeros, see \Cref{rk:witness}.

Note that as a linear projection of a polyhedral cone is a polyhedral cone \cite[Theorem 19.3]{rockafellar-1970a},  $\mC_\mS^\sigma$ is also a finite union of  relative interiors of    polyhedral cones in $\R^{m}$.

\begin{remark}\label[remark]{rk:opposite}
  A computational simplification arises from the observation  that only ``half'' of the pairs  $(\mS,\sigma)$ of a sign matrix and orientation  need to be considered when applying \Cref{thm:mono,thm:full}. To see this, let $\sigma^{-}$ be the orientation obtained by interchanging the two rows of $\sigma$.
  Then, $\mC_{\mS}^\sigma =\varnothing$ if and only if $\mC_{-\mS}^{\sigma^{-}} =\varnothing$.  Explicitly, there is a bijection
\[ \mC_{\mS}^\sigma \rightarrow \mC_{-\mS}^{\sigma^{-}} \qquad \rho \mapsto -\rho \, .\]
Indeed, we have
\[
(A^{\sigma^{-}}_{-\rho})_{ik}
=
-e^{-\rho_i-\rho_{s+k}}(A^\sigma_{\rho})_{ik}\,.
\]
Therefore, $\sign(A^{\sigma^{-}}_{-\rho}) = -\sign(A^{\sigma}_{\rho})$. Also, by \Cref{def:mGsigma}, $\rho\in\mathcal G^\sigma$ if and only if $-\rho\in\mathcal G^{\sigma^{-}}$ (for \eqref{eq:conditionU2}, consider $-z_+$ and $-z_-$). Moreover, the subsets $I^{+}$ and $I^{-}$ are preserved after changing the sign of $\mS$ and swapping the rows of $\sigma$. Under the change $\rho\mapsto-\rho$, the inequalities defining the subsets $\mathcal D_i$ in \eqref{def:ID} are transformed into the corresponding inequalities for the pair $(-\mS,\sigma^{-})$. Hence, if
\[
J=\{i\in I^+\cup I^-:\rho\in\mathcal D_i\}\,,
\]
then the corresponding index set for $(-\mS,\sigma^{-},-\rho)$ is the same $J$, and therefore the condition $\mu\in\Gamma_J$ is preserved as well.

 Additionally, if 
 $(v,\delta,\mu)\in \overline{\mE}_F$, then $(a,x,y)\coloneqq \Phi(v,\delta,\mu) \in \mM_F$ and  
 \Cref{thm:first_characterization} implies that $v=x-y$ and $\delta=\ln x - \ln y$. 
 As also  $(a,y,x) \in \mM_F$,  we have $\Psi(a,y,x)  = (y-x, \ln y - \ln x, \mu')= 
  (-v, -\delta, \mu') \in \overline{\mE}_F$. Hence   $ (-v, -\delta, \mu') \in \overline{\mE}_F$ gives rise to  $(a,y,x)\in \mM_F$. It follows that when finding witnesses of multiple positive zeros as in \Cref{rk:witness}, we do not miss relevant parameter values when considering only one pair  in the set $\{(\mS,\sigma),(-\mS,\sigma^{-})\}$.
\end{remark}

\begin{remark}
As illustrated by \Cref{example:hybrid_histidine_kinase},  witnesses of multiple positive zeros can be chosen to have rational entries if $\rho$ and $\delta$   can be chosen with entries in $\ln(\mathbb{Q}_{>0})$. 
This is   possible provided $L$ is an integer matrix, and $C$ is a rational matrix (as is the case in the application to reaction networks or critical points of polynomials), as any rational polyhedral cone $\mC$ contains points in its relative interior of this form: take generators $\nu_1,\dots,\nu_r$ with integer entries and 
rational $u_1,\dots,u_r > 1$. Then $\sum_{i=1}^r \ln(u_i) \nu_i$ belongs to the relative interior of $\mC$ and has  entries in $\ln(\Q_{>0})$. 
\end{remark}

\section{Additional properties}
\label{section:properties}
We conclude this work by discussing some additional properties that can be explored using the set $\overline{\mE}_F$ from \eqref{eq:Ebar} that encodes multiple positive zeros.

\subsection{Connectivity of the parameter region of multiple positive zeros}
Deciding whether the parameter region where the system has  multiple positive zeros is connected has been a problem of recent interest in chemical reaction network theory \cite{plos:identifying,kaihnsatelek,shiu-connectivity}. 
If  we let $\pi\colon \R^{\bar{m}}_{>0} \times \R^{n}_{>0} \times \R^{n}_{>0} \rightarrow  \R^{\bar{m}}_{>0}$ denote  the projection onto the first factor, then the set 
\[
    \mathcal{A}_F \coloneqq \pi(\mathcal{M}_F)
\]
consists of all parameter values $a$ for which $F_{a,b}$ has multiple positive zeros for some $b \in \R^{d-\bar{s}}$. 
The following proposition gives a sufficient condition for $\mathcal{A}_F$ to be connected. 
In order to state it, given an orthant $\mO\subseteq \mO_L$ and sign matrix 
$\mS \in \{-1,0,1\}^{\bar{s} \times \bar{\ell}}$, we define the following subset of $\overline{\mE}_{F}$:
\begin{equation}\label{eq:EFSO} \overline{\mE}_{F,\mS,\mO}  \coloneqq \{(v,\delta,\mu)\in \overline{\mE}_F: \sign(A_{M^\top(\delta)})= \mS,\ \delta\in \mO\}\,. 
\end{equation}
We consider the orthant $-\mO\coloneqq\{-x : x\in \mO\}$.

\begin{proposition}
\label[proposition]{prop:connectivity}
Let $F \in \R[a,b,x^\pm]^d$ be an augmented vertical  system defined by a principal matrix.
Let 
$\mS \in \{-1,0,1\}^{\bar{s} \times \bar{\ell}}$ and $\mO\subseteq \mO_L$ be an orthant. Assume that:
\begin{enumerate}[label=(\roman*)]
    \item $\overline{\mE}_F=\overline{\mE}_{F,\mS,\mO} \cup \overline{\mE}_{F,-\mS,-\mO}$, and  
    \item the set $\overline{\mE}_{F,\mS,\mO}$ is connected.   
\end{enumerate}
Then $\mathcal{A}_F$ is connected.
\end{proposition}
\begin{proof}
Consider the map $\Phi$ from \eqref{eq:Phi}. 
\Cref{rk:opposite} together with assumption (i) implies that
\[
\pi(\Phi(\overline{\mE}_F)) =\pi(\Phi(\overline{\mE}_{F,\mS,\mO}))\,.
\]
As $\Phi$ is continuous on $\overline{\mE}_{F,\mS,\mO}$,
assumption (ii) implies that $\pi(\Phi(\overline{\mE}_F))\subseteq \mathcal{A}_F$ is connected. The proposition will follow after showing that any point in $\mathcal{A}_F=\pi(\mathcal{M}_F)$ is path-connected to a point in $\pi(\Phi(\overline{\mE}_F))$.

By \Cref{thm:first_characterization},  we have that
\[ \mM_F \setminus \Phi(\overline{\mathcal{E}}_F) = \{(a,x,y)\in \mM_F : x_i=y_i\neq 1 \text{ for some }i\in [n] \}\,. \]

Let $(a,x,y)\in\mathcal{M}_F\setminus\Phi(\overline{\mathcal{E}}_F)$ and     $J  \coloneqq \{i \in [n] : \: x_i = y_i \neq 1\}$. 
Define the path $r_i(t)=(1-t)x_i+t$ for each $i \in [n]$ and $t\in [0,1]$. Define
\[
x_i(t)\coloneqq \begin{cases}
r_i(t)& i \in J\\
x_i& i \notin J\,,
\end{cases}
\qquad
y_i(t)\coloneqq \begin{cases}
r_i(t)& i \in J\\
y_i& i \notin J\,,
\end{cases}
\qquad
a_j(t)\coloneqq a_j   \prod_{i \in J}\left(\frac{x_i}{r_i(t)}\right)^{M_{ij}}
\]
for $j\in[\bar m]$. 
Then for all $t\in [0,1]$, 
\[
a(t)\star x(t)^M=a\star x^M,\qquad
a(t)\star y(t)^M=a\star y^M,\qquad
x(t)-y(t)=x-y\,.
\]
It follows that $(a(t),x(t),y(t))\in\mathcal{M}_F$, hence $a(t)\in \pi(\mM_F)$, for every $t\in[0,1]$. 
Additionally $(a(0),x(0),y(0))=(a,x,y)$ and $(a(1),x(1),y(1))\in \Phi(\overline{\mE}_F)$ as $x_i(1)=1$ whenever $x_i(1)=y_i(1)$. Therefore we have obtained the desired path and we conclude that $\mathcal{A}_F$ is connected.
\end{proof}

\subsection{Nondegeneracy}
For a square augmented vertical   system $F$, that is, $d=n$, 
a zero $x$ of $F_{a,b}$  is \term{nondegenerate}
if the Jacobian $J_{F_{a,b}}(x)$ is nonsingular. A standard computation using that $x\in \R^n_{>0}$ gives that $x$ is nondegenerate if and only if 
\begin{equation}\label{eq:jac}
    \det \begin{pmatrix}
        C\diag(a \star x^M)M^{\top}  \\
        L \diag(x)
    \end{pmatrix} \neq 0\,.
\end{equation} 

Nondegeneracy of the zeros is a relevant  property of augmented vertical  systems, which plays an important role when inferring the existence of multiple positive zeros from smaller subsystems, e.g. \cite{craciun-entrapped,feliu:intermediates,joshishiu,banajipanteaMPNE,cappelletti:flow,banaji:boros:hofbauer:2024b}.
In particular, one wants to determine whether $F$ admits multiple nondegenerate zeros, meaning that  there exists a choice of parameters $(a,b)$ such that $F_{a,b}$ has at least two nondegenerate distinct zeros in $\R^n_{>0}$. 
This is not guaranteed by the existence of multiple positive zeros, as the next example illustrates.
 
\begin{example}
Consider the square augmented vertical   system $F=(a_1 x_1 -a_2x_1+a_3x_1x_2,x_2-b)$ with defining matrices
\[ C = \begin{pmatrix}
    1 & -1 & 1
\end{pmatrix},  \qquad M = \begin{pmatrix}
   1 & 1 & 1 \\
   0 & 0 & 1 
\end{pmatrix}, \qquad  L = \begin{pmatrix}
    0 &  1
\end{pmatrix} \, ,  \] 
where $n=2$, $\bar m=3$, $\bar s=1$, and $\bar\ell=2$, which arises from 
 the following reaction network from \cite[Example 2.1]{Kaihnsa_Nguyen_Shiu_2024}:
 \[ 0 \ce{<-[$a_1$]} X_1 \ce{->[$a_2$]}  2X_1 \qquad X_1 + X_2 \ce{->[$a_3$]} X_2\,. \] 
An easy computation gives that for $a,b>0$, 
$F$ has  no positive zero if $a_1-a_2+a_3b\neq 0$ and infinitely many, all degenerate, if $a_1-a_2+a_3b= 0$.
Hence $F$ admits  multiple positive zeros but not multiple nondegenerate positive zeros.

For illustration, we can see that $F$ admits multiple positive zeros using   \Cref{cor:full}. We have $\mO_L=(\R\setminus \{0\}) \times \{0\}$
and the reduced matrix  is 
\[ \bar{P}=  \begin{pmatrix}
   1 & -1 
\end{pmatrix}\,,  \] 
which induces a forest.
For $\mS = \begin{pmatrix}
    0 & 0
\end{pmatrix}$, we have $I^+\cup I^-=\varnothing$ and
\[ \mC^{\sigma_+}_{\mS} = \big\{ \rho \in \R^3 \colon  \rho_1 = \rho_2 = \rho_3 = \delta, \ \delta\in \R\setminus \{0\} \big\} \neq \varnothing\,.
\]
\cref{cor:full} tells us that the system admits multiple positive zeros. 
\end{example}

A natural question is whether the existence of parameter values and pairs of distinct nondegenerate zeros can be inferred from our setup.
Recall the matrix $\widehat{P}$ from \eqref{eqP}, whose columns form a basis of $\ker(C)$.
For an orthant $\mO \subset \mO_L$ and a sign matrix $\mathcal{S} \in \{-1,0,1\}^{\bar{s} \times \bar{\ell}}$, recall the set $\overline{\mE}_{F,\mS,\mO}$ from \eqref{eq:EFSO} and the functions $x,y$ defining $\Phi$ in \eqref{eq:Phi}. Let $J_\mO  \coloneqq \{i \in [n] : \: \delta_i=0 \text{ for all }\delta\in \mO\}$.  For $t \in \mathbb{R}^{J_{\mO}}_{>0}$, we consider the map $\Pi^{t}_{\mS,\mO}\colon \overline{\mE}_{F,\mS,\mO} \rightarrow \R$ defined by
\begin{equation}\label{eq:Pi}
\Pi^{t}_{\mS,\mO}(v,\delta,\mu)=
\det\begin{pmatrix}
    C\diag( (\widehat{P}\mu) \star e^{M^{\top}(\delta)})M^{\top} \\
    L\diag(x(v,\delta,t))
\end{pmatrix}
\det\begin{pmatrix}
    C\diag(\widehat{P}\mu)M^{\top} \\
    L\diag(y(v,\delta,t))
\end{pmatrix}\,,
\end{equation}
where $x_i(v,\delta,t) = y_i(v,\delta,t) = t_i$ when $i \in J_{\mO}$, and $x_i(v,\delta),y_i(v,\delta)$, otherwise.  For fixed $\mS$ and $\mathcal O$, set
$q_j=e^{\delta_j}$. Since $M$ has integral entries, the equality conditions defining
$\overline{\mathcal E}_{F,\mS,\mathcal O}$ become Laurent polynomial
equations in $(v,q,\mu)$.
The sign and orthant conditions define relatively open subsets of this
algebraic set. Decomposing its algebraic components into smooth loci and
lower-dimensional singular loci, recursively, yields connected
real-analytic strata on which $\Pi^t_{\mS,\mathcal O}$ is a real
analytic function.

\begin{theorem}\label{thm:nondegeneracy}
Let $F =(C(a \star x^M),\; Lx - b) \in \mathbb{R}[a, b, x^{\pm}]^n$ be a square augmented vertical system defined by a principal matrix.
The following statements are equivalent:
\begin{enumerate}[label=(\roman*)]
    \item $F$ admits multiple nondegenerate positive zeros.
    \item There exist a sign matrix $\mathcal{S} \in \{-1, 0, 1\}^{\bar{s} \times \bar{\ell}}$, an orthant $\mathcal{O} \subset \mathcal{O}_L$, and $t \in \mathbb{R}^{J_{\mO}}_{>0}$, such that $\Pi^{t}_{\mathcal{S},\mathcal{O}}$ attains a nonzero value.
    \item There exist a sign matrix $\mathcal{S} \in \{-1, 0, 1\}^{\bar{s} \times \bar{\ell}}$, an orthant $\mathcal{O} \subset \mathcal{O}_L$, and a nonempty open subset $V \subset \mathbb{R}^{J_{\mO}}_{>0}$ such that $\Pi^{t}_{\mathcal{S},\mathcal{O}}$ attains nonzero values on a Euclidean open dense subset of a  connected real-analytic stratum of $\overline{\mE}_{F,\mS,\mO}$ (in its relative topology) for all $t \in V$.
\end{enumerate}
\end{theorem}

\begin{proof}
Let $(a,x,y)\in \mM_F$, $(v,\delta,\mu)\coloneqq\Psi(a,x,y)\in \overline{\mE}_{F}$,
and $\mO$ be the orthant containing $\delta$. For $t \in \mathbb{R}^{J_{\mO}}_{>0}$, let $x(t)\coloneqq x(v,\delta,t)$, $y(t) \coloneqq y(v,\delta,t)$, and
\[
a_j(t)\coloneqq a_j \prod_{i \in J_{\mO}} \left(\frac{x_i}{t_i}\right)^{M_{ij}} \qquad j\in[\bar m]\,.
\]
It holds that 
\[ a(t)\star x(t)^M = a\star x^M\,,\qquad a(t)\star y(t)^M = a\star y^M\,,\quad\text{and}\quad  x(t)-y(t)=x-y \,,\]
and hence $(a(t),x(t),y(t))\in\mM_F$. 
Additionally, by letting $t^*_i=x_i$ if $i\in J_\mO$, we have $(a,x,y)=(a(t^*),x(t^*),y(t^*))$.

The second determinant in \eqref{eq:Pi} is nonzero if and only if $y(t)$ is a nondegenerate  zero of $F_{a(t),Ly(t)}$ by \eqref{eq:jac}, since
\[
a(t)\star y(t)^M = \widehat{P}\mu
\]
by definition of $\Phi$. The first determinant is nonzero if and only if $x(t)$ is nondegenerate, since
\[
a(t) \star x(t)^M
=
(\widehat{P}\mu)\star \left(\frac{x(t)}{y(t)}\right)^M
=
(\widehat{P}\mu)\star (e^\delta)^M
=
(\widehat{P}\mu)\star e^{M^{\top}(\delta)}.
\]
With this, if (i) holds, say $(a,x,y)\in \mM_F$ with $x,y$ nondegenerate, then $\mO$, the corresponding sign matrix $\mS$ and $t^*$ satisfy (ii). Conversely, if (ii) holds, that is $\Pi^{t}_{\mS,\mO}(v,\delta,\mu)\neq 0$, then applying the construction to $(a,x,y)=\Phi(v,\delta,\mu)$ and
using that $\Psi(\Phi(v,\delta,\mu))=(v,\delta,\mu)$, we have that $(a(t),x(t),y(t)) \in \mM_F$ with $x(t),y(t)$ nondegenerate. 
This shows the equivalence between (i) and (ii).
\smallskip

Assume (ii) and choose
$(v_0,\delta_0,\mu_0)\in\overline{\mE}_{F,\mS,\mO}$ and
$t_0\in\R_{>0}^{J_{\mO}}$ such that
\[
\Pi^{t_0}_{\mS,\mO}(v_0,\delta_0,\mu_0)\neq 0\,.
\]
Let $Z$ be a connected real-analytic stratum of $\overline{\mE}_{F,\mS,\mO}$ containing
$(v_0,\delta_0,\mu_0)$. By continuity, there exists an open neighborhood
$V\subset\R_{>0}^{J_{\mO}}$ of $t_0$ such that
\[
\Pi^t_{\mS,\mO}(v_0,\delta_0,\mu_0)\neq 0
\qquad \text{for all } t\in V\,.
\]
Thus, for every $t\in V$, the restriction of
$\Pi^t_{\mS,\mO}$ to $Z$ is not identically zero. Since this restriction
is real analytic, the identity principle
\cite[Section~6.5]{krantz} implies that its zero set has empty interior
in $Z$. Hence $\Pi^t_{\mS,\mO}$ is nonzero on an open dense subset of
$Z$. This proves (iii), while the converse is immediate.
 \end{proof}

\bibliographystyle{alpha}
\bibliography{mybib}

\Addresses

\end{document}